\RequirePackage{fix-cm}

\documentclass{article}
\usepackage[left=3cm, right=3cm, top=2cm,bottom=3cm]{geometry}

\usepackage{graphicx}

\usepackage[utf8]{inputenc}
\usepackage[T1]{fontenc}

\usepackage{amsthm}
\usepackage{amssymb}
\usepackage{amsmath}
\usepackage{verbatim}
\usepackage{mathrsfs}
\usepackage[linesnumbered, ruled]{algorithm2e}
\usepackage{hyperref}
\usepackage[misc]{ifsym}
\usepackage{placeins}
\DeclareMathOperator*{\argmax}{arg\,max}

\usepackage{hyperref}

\usepackage{array}
\newcolumntype{M}[1]{>{\centering\arraybackslash}m{#1}}

\newtheorem{proposition}{Proposition}
\numberwithin{proposition}{section}
\newtheorem{theorem}[proposition]{Theorem}
\newtheorem{lemma}[proposition]{Lemma}%
\newtheorem{corollary}[proposition]{Corollary}
\theoremstyle{remark}
\newtheorem{remark}[proposition]{Remark}
\theoremstyle{definition}
\newtheorem{definition}[proposition]{Definition}
\newtheorem{example}[proposition]{Example}

\usepackage{todonotes}

\newcommand{\Dominik}[1]{{\todo[inline, color=red!40]{DH: #1}}}

\begin{document}

\title{An application of sparse measure valued Bayesian inversion to acoustic
	sound source identification}

\author{
Sebastian Engel \footnote{University of Graz, Institute of Mathematics and Scientific Computing, Heinrichstraße 36, 8010 Graz, Austria, sebastian.engel@uni-graz.at, ORCID: 0000-0003-3968-8108},
~
Dominik Hafemeyer \footnote{Technical University of Munich, Chair for Optimal Control, Department of Mathematics, Boltzmannstr.~3, 85748 Garching b. M\"unchen, Germany, dominik.hafemeyer@tum.de, ORCID: 0000-0002-9161-4811},
~
Christian Münch \footnote{Technical University of Munich, Chair for Numerical Mathematics / Control Theory, Department of Mathematics, Boltzmannstr.~3, 85748 Garching b. M\"unchen, Germany, chris\_muench@mytum.de, ORCID: 0000-0002-9132-3897},
~
Daniel Schaden \footnote{Technical University of Munich, Chair for Numerical Mathematics, Department of Mathematics, Boltzmannstr.~3, 85748 Garching b. M\"unchen, Germany,schaden@ma.tum.de, ORCID: 0000-0003-2419-2839}
}



\maketitle

\begin{abstract}
In this work we discuss the problem of identifying sound sources from pressure measurements with a Bayesian approach. The acoustics are modelled by the Helmholtz equation and the goal is to get information about the number, strength and position of the sound sources, under the assumption that measurements of the acoustic pressure are noisy. We propose a problem specific prior distribution of the number, the amplitudes and positions of the sound sources and algorithms to compute an approximation of the associated posterior. We also discuss a finite element discretization of the Helmholtz equation for the practical computation and prove convergence rates of the resulting discretized posterior to the true posterior. The theoretical results are illustrated by numerical experiments, which indicate that the proven rates are sharp.
\end{abstract}

\section{Introduction}
\label{intro}
Throughout the paper, unless noted otherwise, all functions and measures are complex valued. Let $D \subseteq \mathbb{R}^d$, $d = 2,3$, be an bounded convex polygonal/polyhedral domain. Furthermore let the boundary be in the form $\partial D =: \Gamma =\Gamma_N \cup \Gamma_Z$ where $\Gamma_N := \bigcup_{j=1}^J \overline{\Gamma_j}$ can be written as the union of some faces of $\Gamma$ and $\Gamma_Z := \Gamma \backslash \Gamma_N$. We assume that $\Gamma_Z$ has positive $d-1$-dimensional Hausdorff measure. We model the acoustic pressure $y_u$ as the solution of
\begin{align} \label{eq:helmholtzeq} 
\begin{matrix}\left\{
\begin{matrix}
-\Delta y_u-\left( \frac{\zeta}{c} \right)^2 y_u & = \tau(u) & \text{ in }D,  \\
\partial_\nu y_u -i\frac{ \zeta \rho}{\gamma_{\zeta}} y_u & =0 & \text{ on }\Gamma_Z, \\
\partial_\nu y_u & =g & \text{ on }\Gamma_N.
\end{matrix}
\right.
\end{matrix}
\end{align}
The Helmholtz equation describes a stationary wave and is not elliptic, making some of the analysis atypical. It appears in the modelling of acoustics or electromagnetism (see for example \cite[Chapter 2, §8, Section 7]{dautray2012mathematical} or \cite[Chapters VI, XXV]{Kosljakov1964} for a more physics based approach). To be more specific, $\zeta$ denotes the angular frequency and $c$ the speed of sound. $g\in H^\frac{1}{2}(\Gamma_N)$ is a complex function which models the amplitude of a outside sound source on a part of the boundary and $\tau(u)$ is a complex linear combination, possibly even a series, of Dirac measures with support in $D$. The Dirac measures model sound sources, which we want to identify. $\rho$ describes the density of the fluid and $\gamma_{\zeta}\in \mathbb{C}$ is the wall impedance given by 
\begin{align}
\begin{matrix}
\gamma_{\zeta}:=\beta_{\zeta}+\frac{\alpha_{\zeta}}{\zeta}i.
\end{matrix}
\end{align}
The frequency-dependent material constants $\alpha_{\zeta}>0$ and $\beta_{\zeta}>0$ are related to the viscous and elastic response of the isolating material. However, in what follows, we will assume only that $\beta_{\zeta}\neq 0$. The boundary condition on $\Gamma_Z$ allows the modelling of an absorbing viscoelastic material covering the boundary walls. The boundary conditions on $\Gamma_N$ models external influence on the acoustic pressure. \\

Our aim is to deduce the distribution of the number, locations and amplitudes of the sound sources $\tau(u)$ from pressure measurements $y := (y_u(z_j))_{j=1}^m \in \mathbb{C}^m$ at finitely many distinct points $(z_j)_{j=1}^m$ in $D$. We define the observation operator
\begin{align} \label{eq:observation_op}
G: \, & \ell^1_\kappa \rightarrow \mathbb{C}^m, \quad u  \mapsto  \left( y_u(z_j) \right)_{j=1}^m
\end{align}
where $\ell^1_\kappa$ is a suitable subset chosen from a sequence space of amplitudes $\alpha$ and positions $x$. Later in  Section \ref{sec:observation_potential}, we analyse under which assumptions and restrictions $G$ is actually well-defined.

We analyse the following inverse problem
\begin{align}
\begin{matrix}\label{eq:problem} 
y = G(u)+ \eta \in\mathbb{C}^m,
\end{matrix}
\end{align}
where the measurement noise is denoted by $\eta \in\mathbb{C}^m$. In particular, we give precise meaning to $\eta$, $y$ and $u$ as random variables.

We study \eqref{eq:problem} with a Bayesian approach which means that we are going to propose a problem specific prior, i.e. a distribution, which models prior knowledge of the distribution of the number, amplitudes and positions of the sound sources. We are then going to deduce the posterior distribution of those quantities, which incorporates our prior knowledge and observations from measurements. For the general principle we refer to \cite{[Stu]} or \cite{sullivan2015introduction} for example.

The inverse problem is challenging from various aspects. In particular we would like to note that $G$ will be non-linear thus making this problem delicate. Apart from the non-linearity of $G$ an additional problem will be that (ignoring boundary conditions) solutions to the Helmholtz equation only depend on the distance to the source. (See the proof of Proposition \ref{prop:Greensfct_exist_regul}, (\ref{eq:proof:besselphi}).) Meaning: taking a single measurement one can infer how far the sound source is away from the measurement point, not necessarily where it is located.\\
This work is structured as follows. In Section \ref{sec:HH_equation} we give a precise meaning to \eqref{eq:helmholtzeq} and prove existence and regularity of solutions to \eqref{eq:helmholtzeq}. We proceed to approximate \eqref{eq:helmholtzeq} by finite elements and prove a priori estimates for the error between the exact acoustic pressure and its finite element approximation. 

In Section \ref{sec:bayesian_approach}, we specify the prior and posterior distribution of the number, amplitudes and positions of the unknown sound sources and show that the posterior distribution is well defined. Based on the theoretical stochastic framework we propose a Sequential Monte Carlo method (SMC) in Section \ref{sec:sampling} to solve \eqref{eq:problem}.  In particular, we propose an algorithm to approximate the posterior distribution of the number, the amplitudes and the positions of the sources. We also present that the discretized observation operator $G_h$ is well-defined and meaningful along with a priori error estimates introduced by the finite element discretization.

Section \ref{sec:experiments} is reserved for numerical results which illustrate the theoretical framework. \\

Before we continue we would like to put this work into perspective. In \cite{[BeGaRo],[PiTaTrW],Hongpeng2018}, deterministic models are considered to recover number, amplitude and positions of the sound sources from measurements of the acoustic pressure. The papers rely on techniques of optimal control theory. While the cost functional in \cite{[BeGaRo]} is smooth, non-smooth and sparsity promoting cost functionals are considered in \cite{[PiTaTrW],Hongpeng2018}. Morover, \cite{[BeGaRo]} includes a discussion of finite element approximations.

Various other discretization techniques were also used in situations similar to ours. In the PDE related work of \cite{Kantas2014}, the authors considered an inverse problem, which is governed by the Navier-Stokes equation. The conditions on the Navier-Stokes equation are considered in such a way, that the solution of the PDE is representable by a Fourier series. Such a representation allows the authors to consider a spectral Galerkin method for which exponential rates in the discretization error are obtained. In case of the Helmholtz equation there exists a similar ansatz, which allows to use a spectral method as well, i.e. the Sinc-Galerkin method. The authors in \cite{Hardwood2014} discuss such a Sinc-Galerkin method for the Helmholtz Equation, where exponential error rates can be shown in case of a smooth solution to the Helmholtz equation. In case of point sound sources, i.e. the forcing function is a linear combination of Dirac measures, the authors \cite{Hardwood2014} can not show that an exponential error rate is obtainable. This is still an open question. Furthermore, due to the structure of the Sinc-functions, a more complex geometry, than a square or cube of the domain $D$, requires an overlapping-technique, which is complicated to handle, see \cite{Nancy1995, Nancy1996A}.

Concerning the Bayesian inverse problem, prior distributions with similar structures as in this work can be found in computational statistics: The focus in \cite{Green19955} is to simulate distributions with state spaces of differing dimensions. In particular, \cite{Green19955} considers a measure $\pi$ on the measurable space $(Q,\mathfrak{K})$ where $Q=\bigcup_{k\in \mathfrak{K} } \lbrace k \rbrace\times Q_k$, $Q_k \subseteq \mathbb{R}^k$, $\mathfrak{K}\subseteq\mathbb{N}$. Examples where such distributions can be observed are Bayesian mixture modelling \cite{RISy1997} or non-linear classification and regression \cite{DHMS2002}. We will work with a prior distribution which has its support in an appropriate sequence space $\ell^1_\kappa$. This has the benefit of giving us a function space structure to work with.

In \cite{Knuth1} and \cite{Knuth2} the authors consider the finite dimensional problem
\begin{align}\label{ffs}
	\begin{matrix}
	x(t) = As(t) \in \mathbb{R}^n
	\end{matrix}
\end{align}
for $n\in \mathbb{N}$, where $x_i$ represents the time dependent amplitude recorded by the microphone $i$. The function $s_i$ models the time dependent amplitude of the source and the unknown square matrix $A \in \mathbb{R}^{n\times n}$ is called the mixing matrix. The authors assume that the emitted sound is directly detected by the microphones without any delay and that the mixing is linear. The aim is to obtain an explicit representation of the posterior distribution of the mixing matrix $A$.
Using Bayes' Theorem and several assumptions, the authors are able to show that this problem is well posed and are able to approximate the posterior distribution of the matrix $A$ and also the distance of the sources from the microphones. Comparing the our work and \cite{Knuth1,Knuth2} we allow that the number of detectors is arbitrary, the sound source number is unknown and the prior w.r.t. the positions is more general. The authors in \cite{Knuth1,Knuth2} use the whole $\mathbb{R}^d$ as domain for the position of the sound sources and do not take reflections or exterior influences of the pressure into account, which is possible using a PDE based approach. The focus of \cite{Knuth1,Knuth2} is different and is concentrating on finding the mixing matrix, which does not need the introduction of a PDE depending model.

Regarding the use of the SMC see the works of \cite{Kantas2014, Beskos2015, Chopper2002}. The observation operator in \cite{Kantas2014} is linear and the prior is a Gaussian measure on an infinite dimensional space. In this setting the authors present a standard SMC algorithm, which is based on a previous work of \cite{Chopper2002} and obtains a convergence rate of $\mathcal{O}(1/N)$ with $N$ as the number of particles that approximate the posterior measure \cite{Beskos2014Stable}. This SMC method is also used in \cite{Beskos2015} for an elliptic inverse problem, with a non-linear observation operator and a non-Gaussian prior, which is based on a Fourier series formulation with deterministic basis multiplied by randomized scalar valued random variables of certain kind. The authors want to identify the diffusion coefficent in the elliptic partial differential equation, where fixed Dirac measures act as forcing functions in the differential equation. These Dirac measures are approximated by smoother functions.

Due to the constructed prior, the authors of \cite{Beskos2015}, obtain a bounded likelihood function, which is bounded from below by a $C>0$. This allows the authors to show that the SMC method converges in a specific distance function, see \cite[Theorem 3.1]{Beskos2015}. In our work, we considered a non-linear observation operator as well as a prior which is not Gaussian. Furthermore, we can not directly use the proof of \cite[Theorem 3.1]{Beskos2015} to show convergence in our setting. This is due to the likelihood function, which can obtain values that are arbitrarily near to $0$ in our situation. However, we are still able to show that the SMC method converges in the mean square error with the rate $\mathcal{O}(1/N)$. 

More application-oriented research on the location of sound sources with a Bayesian approach can be found in \cite{Asano2013}. Here, the Gaussian prior and the finite dimensional observation operator are based on a frequency decomposition of the signals. In particular, the authors use a Markov chain Monte Carlo method and compare the simulated results to real experiments. Due to the different structure of the observation operator in \cite{Asano2013}, the model does not consider the Helmholtz equation or its discretization. 
Further practical results are the subject of \cite{Nakamura2009}, where a similar problem with a dynamically moving robot is studied.

\section{The Helmholtz Equation} \label{sec:HH_equation}

\subsection{Preliminaries} \label{sec:prelimHH}
For convenience of the reader we recall some basic results on Sobolev spaces and auxiliary results. We also introduce the notation for the rest of the work. Most definitions are not given in full generality, but can be easily generalized. Moreover, we define several solution concepts to the Helmholtz equation in this section. 

The absolute value of a complex number is denoted by $|\cdot|$. We write $\| \cdot \|$ for the Euclidean norm in $\mathbb{R}^d$. We define the distances for $x\in \mathbb{R}^d$ and $D_0,D_1\subseteq \mathbb{R}^d$
\begin{align*}
\operatorname{dist}(x, D_1) := \inf_{y \in D_1} \|x-y\| \qquad \text{ and } \qquad \operatorname{dist}(D_0, D_1) := \inf_{x \in D_0, y \in D_1} \|x-y\|.
\end{align*}

For $p\in [1,\infty]$ and $k\in\mathbb{N}$ we denote by $W^{k,p}(D)$ complex valued Sobolev spaces, i.e. $L^p(D)$ function whose derivatives of up to order $k$ also lie in $L^p(D)$. They are equipped with the usual norms and thus Banach spaces. Also note that $H^k(D) := W^{k,2}(D)$ is a Hilbert space. These statements and all the following statements of Section \ref{sec:prelimHH} can be found for the real valued case in \cite{Adams1975}. The statements for complex valued functions follow by the decomposition into real and imaginary part.

For the next theorem we need the notion of a Lipschitz domain. A Lipschitz domain is a domain $D$ in $\mathbb{R}^d$ such that locally the boundary of $D$ can be written as the graph of a Lipschitz map. A precise definition is given in \cite[Definition 2.4]{wloka}. We emphasize that any domain appearing in this paper, for example $D$ or $D\setminus B_\kappa(x)$ for $x\in D, \kappa \geq 0$, are Lipschitz domains.

\begin{theorem}[Sobolev Embedding Theorem] \label{thm:sobolevembedding}
For each Lipschitz domain $E\subseteq \mathbb{R}^d$ we have the continuous embedding
\begin{align*}
H^2(E) \hookrightarrow C(\bar E).
\end{align*}
More precisely, there exists a $C>0$ such that
\begin{align*}
\lVert f \rVert_{C(\bar E)} \leq C \lVert f \rVert_{H^2(E)}, \quad \text{for all } f\in H^2(E).
\end{align*}
\end{theorem}


In general, functions in $H^k(D)$ are not well-defined on the boundary, but their boundary values can still be defined in the sense of traces (see \cite[Theorem 5.22]{Adams1975}) for which we write $\cdot|_{\Gamma}$.
%

\begin{definition}
We define
\begin{align*}
H^\frac{1}{2}(\Gamma_N) := \left\lbrace f|_{\Gamma_N}: f\in H^1(D) \right\rbrace
\end{align*}
and recall that for each $f\in H^1(D)$ one has
\begin{align} \label{eq:traceest}
\lVert f|_{\Gamma} \rVert_{H^\frac{1}{2}(\Gamma)} \leq C \lVert f \rVert_{H^1(D)}
\end{align}
with a $C$ independent of $f$.
\end{definition}

In the following, we introduce different solution concepts for the Helmholtz equation.
%
%
\begin{definition}[Weak solution] \label{Def:Weaksolution}
The function $y \in H^1(D)$ is called weak solution to the Helmholtz equation if it satisfies
\begin{align}\label{eq:weakSol}
\int\limits_D \nabla y \cdot \nabla \overline{v} - \left(\frac{\zeta}{c}\right)^2 y \overline{v} ~ dx - \frac{i \zeta \rho}{\gamma_{\zeta}}\int\limits_{\Gamma_Z} y\overline{v}dS(x) 
= \int\limits_D \tau(u) \overline{v} dx + \int\limits_{\Gamma_N} g \overline{v}dS(x) 
\end{align}
for all $v \in H^1(D)$. 
\end{definition}
\begin{definition}[Very weak solution] \label{def:veryweaksol}
The function $y \in L^2(D)$ is called very weak solution to the Helmholtz equation if it satisfies
\begin{align*}
\int\limits_D y (- \Delta \overline{v}) - \left(\frac{\zeta}{c}\right)^2 y \overline{v} ~ dx 
= \int\limits_D \tau(u) \overline{v} dx + \int\limits_{\Gamma_N} g \overline{v}dS(x) 
\end{align*}
for all $v \in H^2(D)$ with $\partial_{\nu} v + i\frac{\zeta \rho}{\bar \gamma_\zeta}v = 0$ on $\Gamma_Z$ and $\partial_{\nu} v = 0$ on $\Gamma_N$.
%
\end{definition}
Both notions of weak and very weak solutions are derived from testing \eqref{eq:helmholtzeq} with a test function and using integration by parts. In case of a very weak solution of the Helmholtz equation with $\tau(u) = \sum_{\ell = 1}^\infty \alpha_\ell \delta_{x_\ell}$ the first integral on the right has to be understood as
\begin{equation*}
\int\limits_D \tau(u) \overline{v} dx = \sum_{\ell = 1}^\infty \alpha_\ell \overline{v(x_\ell)}.
\end{equation*}
This expression is well-defined for $v \in H^2(D)$ since $H^2(D)$ embeddeds into $C(\bar D)$ (see Theorem \ref{thm:sobolevembedding}).

%

The very weak solution to the Helmholtz equation has singularities at the sources $x_{k} \in D$ if $\tau(u) = \sum_{k = 1}^\infty \alpha_k \delta_{x_k}$ and $\alpha_{k}\neq 0$ (see the proof of Proposition \ref{prop:Greensfct_exist_regul}; in particular \eqref{eq:proof:besselphi}). This makes it complicated to estimate the norm of the very weak solution on the entire domain $D$. 
Therefore, we restrict the possible source locations $x$ to a source domain and the given measurement points $z \in D$ to a measurement domain and demand a positive distance of the domains to each other.

\begin{definition}[Source and measurement domain]
For $\kappa > 0$ the source domain $D_\kappa \subseteq D$ is a set satisfying $\operatorname{dist}(D_\kappa, \Gamma) > \kappa$. The measurement domain is defined as 
\begin{equation*}
M_\kappa := \{ x \in D \, | \, \operatorname{dist}(x, D_\kappa) > \kappa \text{  and  } \operatorname{dist}(x, \Gamma) > \kappa \}.
\end{equation*}
\end{definition}
With restriction to these domains we can derive estimates of the norms $\| \cdot \|_{H^2(M_\kappa)}, \| \cdot \|_{W^{2, \infty}(M_\kappa)}$ or point evaluations of the form $y_u(z)$ for $z \in M_\kappa$ with sources in $D_\kappa$. 
We further allow for an infinite amount of sources. 
\begin{definition}[Source set]
The space $\ell^1(\mathbb{C} \times \mathbb{R}^d)$ is the space of sequences $(\alpha_\ell, x_\ell)_{\ell = 1}^\infty \subseteq \mathbb{C} \times \mathbb{R}^d$ such that 
\begin{equation*}
\| (\alpha_\ell, x_\ell)_{\ell = 1}^\infty \|_{\ell^1(\mathbb{C} \times \mathbb{R}^d)} := \sum_{\ell = 1}^\infty (|\alpha_{\ell}| + \| x_\ell \|) < \infty.
\end{equation*}
We denote this Banach space by $\ell^1$. We define the set $\ell^1_\kappa$ as the restriction of $\ell^1$ to the source domain
\begin{equation*}
\ell^1_\kappa := \{(\alpha_\ell, x_\ell)_{\ell = 1}^\infty \in \ell^1 \, | \, x_\ell \in D_\kappa, \quad \text{for all } \ell \in \mathbb{N} \}.
\end{equation*}
Moreover, we introduce the mapping $\tau$ which maps intensities and sources to a series of dirac measures 
\begin{equation*}
\tau((\alpha_\ell, x_\ell)_{\ell = 1}^\infty) := \sum_{\ell = 1}^\infty \alpha_\ell \delta_{x_\ell}.
\end{equation*}
\end{definition}
The general assumption is now that $D_\kappa$ and $\ell^1_\kappa$ are non-empty. The definition of $\ell^1_\kappa$ requires that $0 \in D_\kappa$, otherwise $\ell^1_\kappa$ is empty. This assumption can be bypassed by an appropriate shift of the domain. Thus we assume w.l.o.g. that $0 \in D_\kappa$.

We now study the very weak solution $y_u$ of the Helmholtz equation. We proceed as in \cite[Section 3.2]{[BeGaRo]} to prove existence, uniqueness and norm estimates with respect to the data. Furthermore, we are interested in sensitivity properties of $y_u$ with respect to $u$. 
Later, we want $u$ to follow some prior probability distribution.
This requires that all appearing constants in this section are independent of $u$. 
After the results for the infinite dimensional problem have been established, we continue to prove similar results for solutions $y_{u,h}$ of a discretized problem and show convergence rates.
\subsection{Homogeneous Neumann Data and a Single Sound Source}
We first consider the Helmholtz equation \eqref{eq:helmholtzeq} with a single source source and no outside noise, i.e. $\tau(u) = \delta_x$ and $g=0$. Here $x\in D_\kappa$ and $u = \left( (1,x), (0,0), (0,0), \dots \right) \in \ell^1_\kappa$.

In the following, we will repeat the main steps from \cite[section 3.2]{[BeGaRo]} to show existence and regularity of the very weak solution. We prove both at the same time using a decomposition argument. The main reasons why we not just refer to \cite[section 3.2]{[BeGaRo]} are convienience of the reader and the fact that we need to prove local Lipschitz continuity of the Helmholtz equation with respect to its forcing term (e.g. in Section \ref{sec:bayesian_approach}), which is not done in \cite[section 3.2]{[BeGaRo]}.

\begin{proposition} \label{prop:uniqueness}
Very weak solutions to the Helmholtz equation are unique.
\end{proposition}

\begin{proof}
Because the Helmholtz equation, and in particular its very weak formulation, are (affine) linear in the right hand side, it is sufficient to prove that $g=0, \tau(u) = 0$ imply that any very weak solution satisfies $y = 0$. Let $y$ be such a solution. We will use the so-called transposition method. Consider the following linear map
\begin{align*}
	\mathcal{D}_a \colon \mathcal{D} \rightarrow L^2(D), ~ v \mapsto -\Delta v -\left( \frac{\zeta}{c} \right)^2 v
\end{align*}
with the vector space
\begin{align*}
	\mathcal{D}:=
	\left\lbrace
	v\in H^2(D) :
		\partial_\nu v -i\frac{ \zeta \rho}{\gamma_{\zeta}} v=0 \text{ on }\Gamma_{Z}	\text{ and }\partial_\nu v=0\text{ on }\Gamma_{N}
	\right\rbrace.
\end{align*}
By the continuity of the trace operator this is even a Banach space by virtue of being a closed subspace of $H^2(D)$.

All functions in $\mathcal{D}$ are weak solutions of the Helmholtz equation with right hand side $-\Delta v -\left( \frac{\zeta}{c} \right)^2 v$ and boundary data $g=0$. Due to appropriate $H^2(D)$-estimates for weak solutions of the Helmholtz equation, see \cite[Theorem 3.3, (3.4)]{[BeGaRo]}, we can show that $\mathcal{D}_a$ is a continuous and surjective operator. The uniqueness of weak solutions to \eqref{eq:weakSol}, see \cite[Theorem 3.3]{[BeGaRo]}, implies that $\mathcal{D}_a$ is an isomorphism.

By the definition of $y$ we have for any $v\in \mathcal{D}$:
\begin{align*}
0 = (y,\mathcal{D}_a v)_{L^2(D)} = (\mathcal{D}_a^*y, v)_{\mathcal{D}^*, \mathcal{D}}.
\end{align*}
This entails $\mathcal{D}_a^*y = 0$ and, because adjoints of isomorphims are isomorphisms as well, finally $y=0$.
\end{proof}

\begin{proposition} \label{prop:Greensfct_exist_regul} \label{prop:Greensfct_continuity}
The Helmholtz equation
\begin{align}\label{eq:Greens_fct}
\begin{matrix}\left\{
\begin{matrix}
-\Delta G^x-\left( \frac{\zeta}{c} \right)^2 G^x & = & \delta_x & \text{ in }D,  \\
\partial_\nu G^x -i\frac{ \zeta \rho}{\gamma_{\zeta}} G^x  & = & 0 & \text{ on }\Gamma_Z, \\
\partial_\nu G^x & = & 0 & \text{ on }\Gamma_N.
\end{matrix}
\right.
\end{matrix}
\end{align}
has a unique very weak solution $G^x \in L^2(D)$ for any $x \in D_\kappa$. Additionally there exists a $C_\kappa > 0$ such that
\begin{align*}
\lVert G^x \rVert_{H^2(M_\kappa)}, \lVert G^x \rVert_{W^{2, \infty}(M_\kappa)} \leq C_\kappa \qquad \text{for all } x\in D_\kappa
\end{align*}
and
\begin{align*}
\lVert G^x - G^y \rVert_{H^2(M_\kappa)} \leq C_\kappa \| x-y \| \qquad \text{for all } x,y\in D_\kappa \text{ with } \lVert x-y \rVert < \frac{1}{2}\kappa.
\end{align*}
\end{proposition}

\begin{proof}
The uniqueness follows from Proposition \ref{prop:uniqueness}.

We introduce the fundamental solution of the Helmholtz equation in the whole space, i.e.
\begin{align*}
\begin{matrix}
-\Delta \Phi^x - \left( \frac{\zeta}{c} \right)^2 \Phi^x = \delta_x & \text{in } \mathbb{R}^d.
\end{matrix}
\end{align*}
This equation has to be understood in the sense of (tempered) distributions in $\mathscr{S}(\mathbb{R}^d)^*$, i.e.
\begin{equation} \label{eq:fundamental_solution}
\int_{\mathbb{R}^d} - \Phi^x \Delta \overline{v} - \left( \frac{\zeta}{c} \right)^2 \Phi^x \overline{v} \,dx = \bar v(x) \qquad \text{for all } v \in \mathscr{S}(\mathbb{R}^d).
\end{equation}
The space $\mathscr{S}(\mathbb{R}^d)$ is introduced in \cite[pp. 1058-1061]{zaidler1990nonlinear}. It is not particularly important for our purposes.
%
With this definition the solution of \eqref{eq:fundamental_solution} is given by
\begin{equation} \label{eq:proof:besselphi}
\begin{split}
\Phi^x(z):=\left\lbrace
 \begin{matrix}
\frac{1}{2\pi} Y_0\left( \frac{\zeta}{c} \Vert x-z \Vert \right) & \text{if }d=2, \\ \\
\frac{\exp \left(-\frac{\zeta}{c} \Vert x-z \Vert\right)}{4 \pi \Vert x-z \Vert} & \text{if }d=3.
\end{matrix}
\right.
\end{split}
\end{equation}
Here, $Y_0$ denotes the zero-order second-kind Bessel function \cite[Chapter.~2,~§8.6 Proposition~27]{dautray2012mathematical}. We also refer to \cite[Chapter 1]{BowmanFrank1958ItBf} for some background on Bessel functions. In particular, there holds $\Phi^x \in C^\infty(\mathbb{R}^d\setminus \lbrace x \rbrace)$ and $\Phi^x\vert_D \in L^2(D)$. 

We introduce $p^x$ as the weak solution of
\begin{align*} 
\left\lbrace
\begin{matrix}
-\Delta p^x-\left( \frac{\zeta}{c} \right)^2 p^x &=& 0 & \text{in }D, \\
\partial_\nu p^x - \frac{i\zeta \rho}{\gamma_{\zeta}}p^x &=& -\partial_\nu \Phi^x + \frac{i\zeta \rho}{\gamma_{\zeta}} \Phi^x& \text{on }\Gamma_Z, \\
\partial_\nu p^x &=& - \partial_\nu \Phi^x &\text{on }\Gamma_N.
\end{matrix}
\right.
\end{align*}
For the exact weak formulation see \cite[(3.9)]{[BeGaRo]}, which has additional terms compared to Definition \ref{Def:Weaksolution}, since the boundary condition for $\Gamma_Z$ is not zero. Its exact form is not important for our purposes. According to \cite[Theorem 3.3]{[BeGaRo]} and the remarks before and after \cite[(3.10)]{[BeGaRo]} the solution $p^x$ exists, is unique and $p^x \in H^2(D)$. Moreover, it satisfies
\begin{equation} \label{eq:p_x}
\begin{split}
\lVert p^x \rVert_{H^2(D)} \leq C \left[ \left\lVert \frac{i\zeta\rho}{\gamma_{\zeta}}\Phi^x - \partial_\nu \Phi^x \right\rVert_{H^\frac{1}{2}(\Gamma_Z)} + \left\lVert \partial_\nu  \Phi^x \right\rVert_{H^\frac{1}{2}(\Gamma_N)} \right].
\end{split}
\end{equation}

This implies that $G^x = p^x + \Phi^x|_D$ is the very weak solution to \eqref{eq:Greens_fct}. By this and \eqref{eq:traceest} we have for any $\kappa > 0$
\begin{equation} \label{eq:proof:localH2estimateforGx}
\begin{split}
\lVert G^x \rVert_{H^2(D \setminus B_\kappa(x) )} \leq \lVert p^x \rVert_{H^2(D \setminus B_\kappa(x) )} + \lVert \Phi^x \rVert_{H^2(D \setminus B_\kappa(x) )} \leq C \lVert \Phi^x \rVert_{H^2(D \setminus B_\kappa(x) )}.
\end{split}
\end{equation}
Note that
\begin{equation*}
	Y_0(r)\sim \log(r),\quad Y^{(k)}_0(r) \sim 1/r^k
\end{equation*}
for $r>0$, $r\rightarrow 0$ and $k\in\mathbb{N}$ by \cite[Chapter 1, (1.13) et seq.]{BowmanFrank1958ItBf}. 
Moreover, for any $k\in\mathbb{N}_0$ we have
\begin{align*}
	\left(\frac{\exp \left(-\frac{\zeta}{c} r\right)}{4 \pi r}\right)^{(k)} \sim \frac{1}{r^k}
\end{align*}
for $r>0$, $r\rightarrow 0$. We can thus conclude that each derivative $\partial^k$ of order $k$ of $\Phi^x$ satisfies
\begin{align} \label{eq:proof:phiest}
|\partial^k \Phi^x(z)| \leq
\begin{cases}
 C |\log(\|x-z\|)| & \text{ if } 2-d-k = 0,\\
 C \|x-z\|^{2-d-k} & \text{ else }.
 \end{cases}
\end{align}
Finally, $|x - z| \geq \kappa$ for $z \in D \setminus B_\kappa(x)$ yields the estimate
\begin{equation*}
\lVert \Phi^x \rVert_{H^2(D \setminus B_\kappa(x) )} \le C \kappa^{-d}
\end{equation*} 
for a constant $C>0$ which is independent of $\kappa$ and $x$. Because $\operatorname{dist}(x, \Gamma) > \kappa$, the right hand side in \eqref{eq:p_x} can be bounded similarly using the same arguments. The estimate in the $W^{2, \infty}(M_\kappa)$-norm follows from \cite[Lemma 3.5]{[BeGaRo]} with $D_0 := M_\kappa$.

To see the Lipschitz continuity of $G^\cdot$ let $x,y\in D_\kappa$. Observe that
\begin{align} \label{eq:proof:meanvalueLipschitzGreen0}
\|G^x - G^y\|_{H^2(M_\kappa)} 
\leq  \|p^x - p^y\|_{H^2(D \setminus B_\kappa(x))} +  \|\Phi^x - \Phi^y\|_{H^2(D \setminus B_\kappa(x))} \leq C \|\Phi^x - \Phi^y\|_{H^2(D \setminus B_\kappa(x))}
\end{align}
by \eqref{eq:p_x} and \eqref{eq:traceest}.

Let $z\in M_\kappa$. Note that for each $\xi$ in the line $[x,y]$ we have:
\begin{align*}
\kappa < \| x - z\| \leq \| x - \xi \| + \| z - \xi \| < \frac{1}{2}\kappa+ \| z - \xi \|
\end{align*}
and thus $\frac{1}{2} \kappa < \| z - \xi \|$. Thus for for any $k \in\mathbb{N}_0$ we have by the mean value theorem and \eqref{eq:proof:phiest}
\begin{align*}
|\partial^k \Phi^x(z) - \partial^k \Phi^y(z)| = |\partial^k \Phi^z(x) - \partial^k \Phi^z(y)| \leq C_{\kappa,k} \lVert x - y \rVert.
\end{align*}

We conclude
\begin{equation} \label{eq:proof:meanvalueLipschitzGreen2}
\|\Phi^x - \Phi^y\|_{H^2(D \setminus B_\kappa(x))} \le C_\kappa \|x - y\|.
\end{equation}
This inserted into \eqref{eq:proof:meanvalueLipschitzGreen0} yields the claim.
\end{proof}

\subsection{Inhomogeneous Neumann Data and Multiple Sound Sources}
In this section we present an existence and uniqueness result for the very weak solution of the Helmholtz equation in the general setting.

\begin{theorem} \label{cor:solve_hh_general} \label{cor:continuity_general_hh}
Let $u \in \ell^1_\kappa$ and $g\in H^\frac{1}{2}(\Gamma_N)$ be given. Then the problem
\begin{align} \label{eq:hh_general} 
\begin{matrix}\left\{
\begin{matrix}
-\Delta y_u -\left( \frac{\zeta}{c} \right)^2 y_u & = & \tau(u) & \text{ in }D,  \\
\partial_\nu y_u -i\frac{ \zeta\cdot \rho}{\gamma_{\zeta}} y_u  & = & 0 & \text{ on }\Gamma_Z, \\
\partial_\nu y_u & = & g & \text{ on }\Gamma_N
\end{matrix}
\right.
\end{matrix}
\end{align}
has a unique very weak solution $y_u \in H^2(M_\kappa) \cap L^2(D)$ which satisfies
\begin{align*}
\lVert y_u \rVert_{H^2(M_\kappa)} \leq C_\kappa \left( \|u \|_{\ell^1} + \lVert g \rVert_{H^\frac{1}{2}(\Gamma_N)} \right)
\end{align*}
with a constant $C_\kappa > 0$ depending on $\kappa$ but not on $u$.  For $u,v\in \ell^1_\kappa$ with $\|u - v\|_{\ell^1} < \frac{1}{2} \kappa$ we find
\begin{equation*}
\|y_u - y_v\|_{H^2(M_\kappa)} \le C_\kappa (\|u\|_{\ell^1} + 1) \|u - v\|_{\ell^1};
\end{equation*}
i.e. $y_{(\cdot)}: \ell^1_\kappa \rightarrow H^2(M_\kappa)$ is a non-linear locally Lipschitz continuous function.
\end{theorem}
\begin{proof}
The uniqueness follows from Proposition \ref{prop:uniqueness}.

Let $u = (\alpha_\ell, x_\ell)_{\ell = 1}^\infty \in \ell^1_\kappa$ and $g\in H^\frac{1}{2}(\Gamma_N)$ be arbitrary. First we show that the very weak formulation of \eqref{eq:hh_general} is well-defined. For a  test function $v \in H^2(D)$ this formulation is given by
\begin{align*}
\int\limits_D y_u (- \Delta \overline{v}) - \left(\frac{\zeta}{c}\right)^2 y_u \overline{v} ~ dx 
= \int\limits_D \tau(u) \overline{v} dx + \int\limits_{\Gamma_N} g \overline{v}dS(x).
\end{align*}
The only critical term is the one containing the series of Diracs. Since $v \in H^2(D)\subseteq C(\overline{D})$ by Theorem \ref{thm:sobolevembedding}, this term can be bounded by
\begin{align*}
\left|\int_{D} \tau(u) \bar v dx\right| = \left|\sum_{\ell=1}^\infty \alpha_\ell \bar v(x_\ell) \right| \le \sum_{\ell = 1}^\infty |\alpha_\ell| \max_{x \in \overline{D}} |\bar v(x)| 
 \leq \|u\|_{\ell^1} \|v\|_{C(\overline{D})} < \infty.
\end{align*}
This shows well-definedness of the very weak formulation. 

The next step is to show existence, uniqueness and the norm bound of the very weak solution. We use linearity of equation \eqref{eq:hh_general} to split up the solution in the form $y_u = y_{u, 0} + y_{0, g}$ where $y_{u, 0}$ and $y_{0, g}$ shall satisfy
\begin{align} \label{eq:hh_decomposition}
\begin{split}
	\left\lbrace
	\begin{matrix}
	-\triangle y_{u, 0} -\left( \frac{\zeta}{c} \right)^2 y_{u, 0} = \tau(u) & \text{ in }D, \\
	\partial_\nu y_{u, 0}-\frac{i\zeta\rho}{\gamma_{\zeta}} y_{u, 0}=0 & \text{ on }\Gamma_Z, \\
	\partial_\nu y_{u, 0} = 0 & \text{ on }\Gamma_N,
	\end{matrix}
	\right. 
	\text{ and } 
	\left\lbrace
	\begin{matrix}
	-\triangle y_{0, g} -\left( \frac{\zeta}{c} \right)^2 y_{0, g} = 0 & \text{ in }D, \\
	\partial_\nu y_{0, g}-\frac{i\zeta\rho}{\gamma_{\zeta}} y_{0, g} = 0 & \text{ on }\Gamma_Z, \\
	\partial_\nu y_{0, g}=g & \text{ on }\Gamma_N.
	\end{matrix}
	\right.
\end{split}
\end{align}

The unique very weak solution of the first problem is given by $y_{u,0} = \sum_{\ell = 1}^\infty \alpha_\ell G^{x_\ell}$ according to Proposition \ref{prop:Greensfct_exist_regul}. We apply the bound in the same proposition and obtain a bound of the very weak solution of the form
\begin{align*}
\lVert y_{u,0} \rVert_{H^2(M_\kappa)} \leq \sum_{\ell=1}^\infty |\alpha_\ell| \lVert G^{x_\ell} \rVert_{H^2(M_\kappa)} \leq C_\kappa \sum_{\ell=1}^\infty |\alpha_\ell| 
\leq C_\kappa \|u\|_{\ell^1},
\end{align*}
where the constant $C_\kappa > 0$ does not depend on $u$.

We apply \cite[Theorem 3.3]{[BeGaRo]} to obtain existence and uniqueness of a weak solution $y_{0, g}$ for the second equation of \eqref{eq:hh_decomposition}, which satisfies
\begin{align*}
\lVert y_{0, g} \rVert_{H^2(D)} \leq C \lVert g \rVert_{H^\frac{1}{2}(\Gamma_N)}.
\end{align*}
Adding $y_{u, 0}$ and $y_{0, g}$ concludes the existence proof.

Now let $u = (\alpha_\ell, x_\ell)_{\ell = 1}^\infty \in \ell^1_\kappa$ and $v = (\beta_\ell, z_\ell)_{\ell = 1}^\infty \in \ell^1_\kappa$ be given with $\|u - v\|_{\ell^1} < \frac{1}{2} \kappa$. We decompose $y_u$ and $y_v$ as in \eqref{eq:hh_decomposition} to see that the difference of the solutions satisfies 
\begin{equation*}
y_u - y_v = y_{u, 0} - y_{v, 0}.
\end{equation*}
This allows us to assume $g = 0$. Thus a straight forward computation shows
\begin{align*}
\|y_u - y_v\|_{H^2(M_\kappa)} &\le \sum_{\ell = 1}^\infty \|\alpha_\ell G^{x_\ell} - \beta_\ell G^{z_\ell}\|_{H^2(M_\kappa)} \\
&\leq \sum_{\ell = 1}^\infty \left[|\alpha_\ell - \beta_\ell| \|G^{x_\ell}\|_{H^2(M_\kappa)} 
+ |\beta_\ell| \|G^{x_\ell} - G^{z_\ell}\|_{H^2(M_\kappa)} \right].
\end{align*}
We use the assumption $\|u - v\|_{\ell^1} < \frac{1}{2} \kappa$ which implies that for any $\ell \in \mathbb{N}$
\begin{align*}
\|x_\ell - z_\ell\| < \frac{1}{2} \kappa \qquad \text{ and } \qquad |\beta_\ell| \le |\alpha_\ell| + |\beta_\ell - \alpha_\ell| < |\alpha_\ell| + \frac{1}{2} \kappa.
\end{align*}
Applying these inequalities and Proposition \ref{prop:Greensfct_exist_regul} for every $\ell \in \mathbb{N}$ we further estimate
\begin{align*}
\|y_u - y_v \|_{H^2(M_\kappa)} 
 & \leq C_\kappa \sum_{\ell = 1}^\infty \left[ |\alpha_\ell - \beta_\ell| + \left(|\alpha_\ell| + \frac{1}{2} \kappa \right) \|x_\ell - z_\ell\| \right] \\
&\le C_\kappa \max_{\ell \in \mathbb{N}} \left\{ 1, |\alpha_\ell| + \frac{1}{2} \kappa \right\}  \sum_{\ell = 1}^\infty ( |\alpha_\ell - \beta_\ell| + \|x_\ell - z_\ell\| ) \\
&\le C_\kappa \left(\max_{\ell \in \mathbb{N}} |\alpha_\ell| + 1 \right) \|u - v \|_{\ell^1},
\end{align*}
which proves the claim.
\end{proof}

\subsection{Finite Element Spaces}
We follow \cite{[BeGaRo]} and discretize the Helmholtz equation  by piecewise linear finite elements. To this we end consider a family of triangulations $\left( \mathcal{T}_h \right)_{h>0}$ of $D$. More precisely, for each $h>0$ the set $\mathcal{T}_h$ consists of closed triangles/tetrahedrons $T\subseteq \mathbb{R}^d$ such that
\begin{align*}
\bar D = \bigcup_{T\in \mathcal{T}_h } T \qquad \text{ for all } h>0.
\end{align*}
The corners of the triangles are called nodes of $D$. We assume that no hanging nodes or hanging edges exist. More precisely, we suppose that for all $h>0$ and any $T,T^\prime\in\mathcal{T}_h$ the intersection $T\cap T^\prime$ is either empty, a single point, a common edge of $T$ and $T^\prime$ or a common facet of $T$ and $T^\prime$. The last case is only relevant for $d=3$. The domain $D$ is polygonal/polyhedral which implies the existence of such a triangulation.
For each triangulation $\mathcal{T}_h$ we write $h_T$ for the diameter of a triangle $T\in \mathcal{T}_h$. The mesh size of $\mathcal{T}_h$ is given by $h:= \max_{T\in\mathcal{T}_h}h_T$. We denote by $\rho_T$ the diameter of the largest ball contained in $T\in \mathcal{T}_h$. We make the following assumption for the remaining part of the paper:
\begin{align*}
\exists \sigma_1, \sigma_2 >0: \frac{h_T}{\rho_T} \leq \sigma_1, ~ \frac{h}{h_T} \leq \sigma_2 \qquad \text{for all } T\in \mathcal{T}_h, \forall h>0.
\end{align*}
This or similar conditions are called shape regularity and quasi-uniformity, as this for example prevents our triangles from becoming too acute or too flat. We note that these conditions imply that 
\begin{align} \label{eq:quasiuniform}
\max_{T\in\mathcal{T}_h} \rho_T \geq \frac{h}{\sigma_1\sigma_2}
\end{align}
which is also sometimes called quasi-uniformity.

	
We associate with each triangulation $\mathcal{T}_h$ the finite element space $V_h$ which consists of globally continuous and piecewise linear funtions:
\begin{align*}
V_h := \left\lbrace v_h \in H^1(D) \cap C(\bar D): v_h|_T T \text{ is affine linear for all } T\in\mathcal{T}_h \right\rbrace.
\end{align*}

\subsection{Galerkin Approximations}

\begin{definition}[Discrete solution]
The function $y_h \in V_h$ is called discrete solution, or Galerkin Approximaton, to the Helmholtz equation 
if it satisfies
\begin{align} \label{eq:discreteHH}
\begin{split}
\int\limits_D \nabla y_h \nabla \overline{v}_h - \left(\frac{\zeta}{c}\right)^2 y_h \overline{v}_h ~ dx - \frac{i \zeta \rho}{\gamma_{\zeta}}\int\limits_{\Gamma_Z} y_h\overline{v}_hdS(x) 
= \int\limits_D \tau(u) \overline{v}_h dx + \int\limits_{\Gamma_N} g \overline{v}_h dS(x) 
\end{split}
\end{align}
for all $v_h \in V_h$. This is essentially the weak formulation \eqref{Def:Weaksolution} with a different test function space. We prove existence and uniqueness of solutions to \eqref{eq:discreteHH} in Theorem \ref{thm:discretized_hh_has_solution}.
\end{definition}

The following proposition allows us to work with $L^2(D)$ functions instead of diracs in the discrete setting. For such a construction also see \cite[Theorem 1 and after Lemma 3]{Scott1973}.
\begin{proposition} \label{prop:deltaxh}
Let $x \in \overline{D}$. There is a $\delta_{x,h} \in V_h$ such that
\begin{align} \label{eq:delta_x_h_point_value}
\int_D \delta_{x, h} \bar{v}_h dx = \bar{v}_h(x), \quad \text{ for all } v_h \in V_h.
\end{align}
Furthermore we have $\lVert \delta_{x,h} \rVert_{L^2(D)} \leq C$ where $C$ does not depend on $x$, but on $h$. The mapping $x \mapsto \delta_{x, h}$ from $\overline{D}$ to $V_h$ is Lipschitz continuous.
\end{proposition}
\begin{proof}
Let $x \in \overline{D}$ and define the functional $j_x \colon V_h \rightarrow \mathbb{C}$, $j_x(v_h) = \bar{v}_h(x)$. Because $V_h$ is finite dimensional it is a Hilbert space together with the inner product $(\cdot, \bar\cdot)_{L^2(D)}$. Thus the Riesz representation theorem implies the existence of a $\delta_{x,h}\in V_h$ so that
\begin{align*}
\bar{v}_h(x) = \int_D \delta_{x,h}, \bar v_h ~dx \qquad \text{for all } v_h\in V_h.
\end{align*}
By construction and the equivalence of norms on $V_h$ we have
\begin{align*}
\lVert \delta_{x,h} \rVert_{L^2(D)}^2 = |\delta_{x,h}(x)| \leq \Vert \delta_{x,h} \rVert_{L^\infty(D)} \leq C \lVert \delta_{x,h} \rVert_{L^2(D)}.
\end{align*}
Note that this $C$ depends on $h$, but not on $x$.

For the Lipschitz continuity let $x_1, x_2 \in \overline{D}$. We have that $\delta_{x_1, h}, \delta_{x_2, h} \in V_h$ and both are Lipschitz continuous with constants $\lVert \nabla \delta_{x_1,h} \rVert_{L^\infty(D)}$ and $\lVert \nabla \delta_{x_2,h} \rVert_{L^\infty(D)}$, see for example \cite[Proposition 2.13]{ambrosio}. This yields
\begin{align*}
 \left|\int_D (\delta_{x_1, h} - \delta_{x_2, h}) \overline{(\delta_{x_1, h} - \delta_{x_2, h})} dx \right| 
 & \leq |\overline{\delta_{x_1, h}(x_1) - \delta_{x_1, h}(x_2)}| + |\overline{\delta_{x_2, h}(x_1) - \delta_{x_2, h}(x_2)}| \\
&\leq \left( \lVert \nabla \delta_{x_1,h} \rVert_{L^\infty(D)} + \lVert \nabla \delta_{x_2,h} \rVert_{L^\infty(D)} \right) \|x_1 - x_2 \|.
\end{align*} 

Again using the equivalency of norms we can use the bound from before: $\lVert \delta_{x_1,h} \rVert_{W^{1,\infty(D)}} \leq C \lVert \delta_{x_1,h} \rVert_{L^2(D)} \leq C$, where $C$ does not depend on $x$, but on $h$. The same estimate holds for $\delta_{x_2,h}$ concluding the proof.
\end{proof}

\begin{proposition} \label{prop:discrete_regular_problem}
There exists a $h_0>0$ such that for all $h\in (0,h_0], f\in L^2(D)$ and $g\in H^{\frac{1}{2}}(\Gamma_N)$ there exists a unique discrete solution $y_{f,h} \in V_h$ of
\begin{align} \label{eq:prop:discrete_regular_problem}
\begin{split}
\int_D \nabla y_{f,h} \nabla \bar v_h - \left( \frac{\zeta}{c} \right)^2 y_{f,h}\bar v_h ~dx - \frac{i\zeta\rho}{\gamma_\zeta} \int_{\Gamma_Z} y_{f,h}\bar v_h ~dS
 = \int_D f\bar v_h~dx + \int_{\Gamma_N} g \bar v_h ~dS
\end{split}
\end{align}
for all $v_h\in V_h$. Furthermore there holds
\begin{equation*}
\|y_{f, h}\|_{L^2(D)} \leq C \left(\|f\|_{L^2(D)} + \|g\|_{H^{\frac{1}{2}}(\Gamma_N)} \right)
\end{equation*}
with $C$ independent of $f$ and $g$.
\end{proposition}

\begin{proof}

We introduce a sesquilinear form on $H^1(D)$
\begin{align*}
\begin{split}
a(y, v) & := \int_D \nabla y \nabla \bar v - \left( \frac{\zeta}{c} \right)^2 y v ~dx 
 - \frac{i\zeta\rho}{\gamma_\zeta} \int_{\Gamma_Z} yv ~dx,
\end{split}
\end{align*}
which can be used to formulate Definitions \ref{Def:Weaksolution} and \ref{def:veryweaksol}. $a$ satisfies G\r{a}rding's inequality by \cite[Lemma 3.2]{[BeGaRo]}, which means that there exist constants $C_1,C_2>0$ such that $\lvert a(q,q) + C_2 \lVert q \rVert_{L^2(D)}^2 \rvert \geq C_1 \lVert q \rVert_{H^1(D)}^2$ holds for any $q\in H^1(D)$. Now one can simply use the same ideas used in the real case in \cite{[Schatz]}.
\end{proof}

Note that we have to use G\r{a}rding's inequality as the Helmholtz equation is non-elliptic.


\begin{theorem} \label{thm:discretized_hh_has_solution}
There exists a $h_0 > 0$ such that for all $h \in (0,h_0]$ the following holds: Let $g \in H^{\frac{1}{2}}(\Gamma_N)$ and $u = (\alpha_\ell, x_\ell)_{\ell = 1}^\infty \in \ell^1_\kappa$ be given. Then the Helmholtz equation has a unique discrete solution $y_{u, h} \in V_h$, that is for all $v_h \in V_h$ it satisfies 
\begin{align} \label{eq:hh_full_discrete}
\begin{split}
\int\limits_D \nabla y_{u, h} \nabla \overline{v}_h - \left(\frac{\zeta}{c}\right)^2 y_{u, h} \overline{v}_h ~ dx - \frac{i \zeta \rho}{\gamma_{\zeta}}\int\limits_{\Gamma_Z} y_{u, h} \overline{v}_hdS(x) 
= \int\limits_D \tau(u) \overline{v}_h dx + \int\limits_{\Gamma_N} g \overline{v}_h dS(x).
\end{split}
\end{align}
Furthermore, the solution is given by
\begin{equation*}
y_{u, h} = \sum_{\ell = 1}^\infty \alpha_\ell G^{x_\ell}_h + y_{g, h},
\end{equation*}
where $y_{g, h}$ is the discrete solution of the Helmholtz equation \eqref{eq:discreteHH} with zero forcing term and Neumann data $g$. $G_h^x$ is the discrete version of the Green's function and solves \eqref{eq:discreteHH} for $g=0$ and $\tau(u) = \delta_{x, h}$.
\end{theorem}

\begin{proof}
By definition we have $(\delta_x, v_h) = (\delta_{x,h}, v_h)$ for any $v_h\in V_h$. Thus by Proposition \ref{prop:discrete_regular_problem}, we see that
\begin{equation*}
y_{u, h} = \sum_{\ell = 1}^\infty \alpha_\ell G^{x_\ell}_h + y_{g, h}
\end{equation*}
exists and is well defined as the series converges in $L^2(D)$ by $\sum_{\ell=1}^\infty |\alpha_\ell| < \infty$ and $\lVert G^{x_\ell}_h \rVert_{L^2(D)} \leq C$ with a $C$ independent of $x$. Linearity of the integral now shows that $y_{u, h}$ is indeed a solution to \eqref{eq:hh_full_discrete}.
\end{proof}	

%

\subsection{Pointwise error estimates}
The following proposition is proven in the appendix for clarity.
\begin{proposition} \label{lem:pointwise_green_disc_error}
There exist $h_0, C_\kappa > 0$ such that the following holds:
Let $z \in M_\kappa$. Then for any $h \in (0,h_0]$ we have
\begin{align*}
|G^x(z)-G^x_h(z)| \leq C_\kappa |\ln h| h^2 \qquad \text{for all } x \in D_\kappa.
\end{align*}
\end{proposition}

\begin{theorem} \label{thm:pointwise_disc_error_estimate}
There exist $h_0, C_\kappa>0$ such that the following holds:
Let $z \in M_\kappa$, $u \in \ell^1_\kappa$ and $g\in H^{\frac{1}{2}}(\Gamma_N)$. Then for all $h\in (0,h_0]$ the discrete solution $y_{u, h} \in V_h$ satisfies
\begin{equation*}
|y_u(z) - y_{u, h}(z)| \leq C_\kappa |\ln h| h^2 \left( \|u\|_{\ell^1} + \lVert g \rVert_{H^\frac{1}{2}(\Gamma_N)} \right).
\end{equation*}
\end{theorem}
\begin{proof}
We are allowed to evaluate $y_u$ and $y_{u, h}$ at $z$ because $y_u \in H^2(M_\kappa)\subseteq C(\bar M_\kappa)$ and $y_{u, h} \in V_h$. For the error estimate observe, with $y_g$ solving \eqref{eq:hh_decomposition} and $y_{g,h}$ being the solution to \eqref{eq:discreteHH} for $\tau(u)=0$,
\begin{align*}
|y_u(z) - y_{u, h}(z)| & \leq \sum_{\ell=1}^\infty |\alpha_\ell| \lvert G^{x_\ell}(z) - G^{x_\ell}_h(z) \rvert   + |y_g(z) - y_{g,h}(z)| \\
& \leq \|u\|_{\ell^1} \sup_{x \in D_\kappa} \lvert G^x(z) - G^x_h(z) \rvert + |y_g(z) - y_{g,h}(z)|.
\end{align*}
We apply Lemma \ref{lem:pointwise_green_disc_error} to bound the term under the supremum uniformly in $x$ and obtain the desired rate of $|\ln h| h^2$ for the first term. \cite[Theorem 4.2]{[BeGaRo]} gives the estimate for the second term. It is stated in such a way that the minimal mesh size may depend on $g$, but checking its proof and the reference to \cite[Theorem 5.1]{SchatzWahlbinMaxNormEst} in the proof of \cite[Theorem 4.2]{[BeGaRo]} shows that it is independent of $g$.
\end{proof}

%

\begin{corollary} \label{cor:discrete_continuity}
There exist $h_0, C_\kappa > 0$ such that the following holds:
Let $z \in M_\kappa$. Then for $h\in (0,h_0]$ and $g \in H^{\frac{1}{2}}(\Gamma_N)$ the solution mapping $y_{(\cdot), h}(z): \ell^1_\kappa \rightarrow \mathbb{C}$ is continuous. More precisely there holds
\begin{align}\label{ineq:ContiOfDiscSol}
\begin{matrix}
\lvert
y_{u, h}(z) - y_{v, h}(z)
\rvert
\leq
C_\kappa \left(
\|u\|_{\ell^1}+1
\right)\|u - v\|_{\ell^1}
\end{matrix}
\end{align}
for $u, v \in\ell^1_\kappa$ with $\|u - v \|_{\ell^1} < \frac{1}{2} \kappa$.
\end{corollary}
\begin{proof}
Let $h_0>0$ be from Theorem \ref{thm:discretized_hh_has_solution} and $h\in (0, h_0]$. For $x_1, x_2 \in D_\kappa$ let $G^{x_1}_h - G^{x_2}_h \in V_h$ be the discrete solution of the Helmholtz equation with right hand side $f = \delta_{x_1,h} - \delta_{x_2,h}$ and Neumann data $g = 0$. We use Proposition \ref{prop:discrete_regular_problem} and Proposition \ref{prop:deltaxh} to get
\begin{align*}
\lVert G^{x_1}_h - G^{x_2}_h \rVert_{L^2(D)} \leq C \|\delta_{x_1, h} - \delta_{x_2, h}\|_{L^2(D)} \leq C\|x_1 - x_2 \|.
\end{align*}
with $C$ independent of $x_1$ and $x_2$.

Theorem \ref{thm:pointwise_disc_error_estimate} shows that $|G^{x_1}_h(z)| \leq C$ with $C$ independent of $x_1$. Hence a similar computation as in the proof of Theorem \ref{cor:continuity_general_hh} replacing the Greens function with its discrete counterpart and replacing $\| \cdot \|_{H^2(M_\kappa)}$ with $|\cdot (z)|$ shows the result.  
\end{proof}

\section{The Inverse Problem} \label{sec:bayesian_approach}

Up to this point all definitions and all the analysis was carried out in a deterministic setting. Now let us recall problem \eqref{eq:problem}. Our aim is to analyse the following inverse problem with a Bayesian approach:
\begin{align*}
\begin{matrix}
y = G(u)+\eta.
\end{matrix}
\end{align*}
The observed data $y \subseteq \mathbb{C}^m$ is fixed. We also recall the definition of the observation operator, namely
\begin{align*}
\begin{matrix}
G: \ell^1_\kappa &\rightarrow & \mathbb{C}^m, \\
u := (\alpha_\ell, x_\ell)_{\ell = 1}^\infty & \mapsto & \left( y_u(z_j) \right)_{j=1}^m.
\end{matrix}
\end{align*}
Mind that $G$ is non-linear as $\tau$ is non-linear. $G$ is well defined by Corollary \ref{cor:solve_hh_general} and the embedding $H^2(M_\kappa)\subseteq C(\bar M_\kappa)$.
In the Bayesian setting we assume that $u$ follows some prior distribution $\mu^0$, which we will introduce in Section \ref{sec:prior}. We then want to incorporate the knowledge from the measured data $y\in\mathbb{C}^m$ to obtain the posterior distribution $\mu^y$. Let us assume that the measurement noise $\eta$ is multivariate complex normal distributed such that $\eta$ has a probability density proportional to
\begin{align*}
\rho_\eta(z) \, \propto \, \exp\left(- \frac{1}{2}\|z\|^2_\Sigma \right), \quad \forall z \in \mathbb{C}^m.
\end{align*}

Here, $\Sigma\in \mathbb{C}^{2m\times 2m}$ is a Hermitian, positive definite complex matrix of the form 
\begin{align*}
\Sigma=\begin{pmatrix}
\Gamma & C \\ \bar{C} & \bar{\Gamma}
\end{pmatrix}
\qquad \text{ and } \qquad 
\| z\|^2_\Sigma := (\bar z^T, z^T) \Sigma^{-1} \begin{pmatrix}
z \\
\bar z
\end{pmatrix}
\end{align*}
We denote such a random variable with mean $m$, covariance $\Gamma$ and relation $C$ as $N(m, \Gamma, C)$.

In Section \ref{sec:posterior} we will introduce the posterior distribution as
\begin{equation*} 
\mu^y(F) = \frac{1}{\Lambda} \int_F \exp\left(- \frac{1}{2}\|y - G(u)\|^2_\Sigma \right) \, d\mu^0(u)
\end{equation*}
for $F \in \mathcal{F}$, where $\mathcal{F}$ is a Sigma algebra on $\ell^1_\kappa$. $\Lambda$ is a normalization constant.

The term $\exp(- \frac{1}{2}\|y - G(u)\|^2_\Sigma)$ can be viewed as penalization for deviating too far away from the observed data. Thus the posterior combines both prior knowledge of the solution and the measurements favouring those $u \in \ell^1_\kappa$ which closely match the observation.\\

Before we continue we recall some basic definitions and tools.
If two measures $\mu^1$ and $\mu^2$ have a Radon-Nikodym derivative with respect to  a third measure $\nu$, then we are able to define a metric between them,
see \cite[Section 6.7]{[Stu]}. 

\begin{definition}[Hellinger Distance]
Let $\mu^1, \mu^2, \nu$ be measures on $X$ such that $\mu^1, \mu^2$ have a Radon-Nikodym Derivative with respect to  $\nu$. Then the Hellinger distance between $\mu^1$ and $\mu^2$ is defined as
\begin{equation}\label{def:HellDist}
d_{\text{Hell}}(\mu^1, \mu^2) = \left(\int_X \frac{1}{2}\left(\left(\frac{d\mu^1}{d\nu}\right)^{\frac{1}{2}} - \left(\frac{d\mu^2}{d\nu}\right)^{\frac{1}{2}}\right)^2 d\nu \right)^{\frac{1}{2}}.
\end{equation}

\end{definition}
The Hellinger distance tells us how well two measures agree.

\subsection{The Prior} \label{sec:prior}
Recall that the Bayesian approach needs prior knowledge in form of a probability distribution which fits to problem \eqref{eq:problem}. We want to recover the number of sound sources, their amplitudes, and their positions in $D_\kappa$. Let $(\ell^1_\kappa, \mathcal{F})$ be the measurable set where $\mathcal{F}$ is the Borel $\sigma$-algebra associated with the open sets generated by the norm $\| \cdot \|_{\ell^1}$ on the open set $\ell^1_\kappa$. For $k \in \mathbb{N}_0$ we define the sets
\begin{equation*}
\ell^1_{k, \kappa} := \{ (\alpha_\ell, x_\ell)_{\ell = 1}^{\infty} \in \ell^1_\kappa \, | \, \alpha_\ell = 0, x_\ell = 0, \quad \text{for all } \ell > k\} \subseteq \ell^1.
\end{equation*}
Using this notation we are able to construct a specific probability measure on $(\ell^1_\kappa, \mathcal{F})$.
\begin{theorem} \label{thm:prior_theorem}
Let $q$ be a probability mass function on $\mathbb{N}_0$ and for every $k \in \mathbb{N}_0$ let $\mu^0_k$ be a probability measure on $(\ell^1_\kappa, \mathcal{F})$ such that $\mu^0_k(\ell^1_{k, \kappa}) = 1$. Then $\mu^0$ defined by
\begin{equation*}
\mu^0(F) = \sum_{k \in \mathbb{N}_0} q(k) \mu^0_k(F), \quad \text{ for all } F \in \mathcal{F},
\end{equation*}
is a well-defined probability measure on $(\ell^1_\kappa, \mathcal{F})$. Furthermore samples from $\mu^0$ have finitely many non-zero entries, that is 
\begin{equation} \label{eq:prior_measure_finitely_many}
\mu^0 (\{u \in \ell^1_\kappa \, | \, u \in \ell^1_{k, \kappa} \text{  for some } k \}) = 1.
\end{equation}
\end{theorem} 
We can verify the theorem by standard arguments.
We now state a proposition from which we can deduce whether $u \sim \mu^0$ has moments, which will later be important to show that the posterior is well-defined.
\begin{proposition}
Assume that $\mu^0$ is constructed as in Theorem \ref{thm:prior_theorem} and that for any $k \in \mathbb{N}$ the measure $\mu^0_k$ has $p$-th moment. Let $\operatorname{id}$ denote the identity on $\ell^1_\kappa$. Then there holds
\begin{equation*}
\mathbb{E}_{\mu^0}[\|\operatorname{id}\|^p_{\ell^1}] = \sum_{k \in \mathbb{N}_0} q(k) \mathbb{E}_{\mu^0_k}[\|\operatorname{id}\|^p_{\ell^1}],
\end{equation*}
and if the last series is finite, then $\mu^0$ has $p$-th moment.
\end{proposition}
Again, we can verify the proposition by standard arguments.

The motivation to choose such a measure follows from property \eqref{eq:prior_measure_finitely_many}. This ensures that the forcing term for \eqref{eq:helmholtzeq} consists $\mu^0$-almost surely of a finite number of Diracs. A practical approach to construct such a measure is to express $u$ in terms of random variables $k, \alpha, x$ such that the Diracs on the right hand side of the Helmholtz equation are given by
\begin{equation} \label{eq:forcing_term}
\tau(u) = \sum_{\ell = 1}^k \alpha^k_\ell \delta_{x^k_\ell}.
\end{equation}
The function $q$ is the probability mass function of the random variable $k$, and given $k$ we are able to define the probability measure $\mu^0_k$ for the intensities $(\alpha^k_1,...,\alpha^k_k) \in \mathbb{C}^k$ and positions $(x^k_1,...,x^k_k) \in D^k_\kappa$. 

Another reason is that we essentially consider a prior measure on (a subspace of) the Banach space of finite Radon-Measures $M(D_\kappa)$, respectively a random variable $\tau(u)$ with values in $M(D_\kappa)$. However, $M(D_\kappa)$ is a non-separable Banach space and thus does not posses a countable basis or a natural scalar product. Therefore standard techniques based on orthogonal decomposition like the Karhunen-Lo\`eve expansion (KLE) for random fields in separable Hilbert spaces cannot be applied. Nevertheless a prior satisfying Theorem \ref{thm:prior_theorem} with $q > 0$ and $\operatorname{supp}(\mu^0_k) = \overline{\ell^1_{k, \kappa}} \cong \mathbb{C}^k \times \overline{D}^k_\kappa$ satisfies
\begin{align*}
\begin{matrix}
\operatorname{Image}(\tau(u)) = \left\lbrace 
\sum\limits_{\ell=1}^k \alpha_{\ell} \delta_{x_\ell} \, | \, \quad k \in \mathbb{N}_0,\text{ }(\alpha_\ell)_{\ell = 1}^{k}\in \mathbb{C}^k,\text{ }(x_\ell)_{\ell=1}^k \in D_\kappa^k
\right\rbrace
\end{matrix}.
\end{align*}
Elements of this set are dense in $M(D_\kappa)$ w.r.t. the $w^*-$topology of $M(D_\kappa)$ (see for example \cite[Remark 2.1]{Bredies2013}). Therefore the random field $\tau(u)$ is a good substitute for a classical KLE on the non-separable Banach space $M(D_{\kappa})$.

Another desirable property of $\tau(u)$ is that samples can be easily generated and their support is sparse in $D_{\kappa}$. As we already know from the field of optimal control, the use of sparse controls from the measure space $M(D_\kappa)$ (i.e. the support of controls is a Lebesgue zero set, see \cite{CaSEMA}) is often desired in application and practice. For example, in our inverse problem a point source is much more realistic than a sound source represented by a regular function, which has a support that is not a Lebesgue zero set.

\subsection{Properties of the observation operator and potential} \label{sec:observation_potential}

Now we investigate the non-linear observation operator $G$. It is important to see that the Helmholtz equation only influences the posterior through the observation operator $G$. Therefore, we derive suitable properties of $G$ from the underlying PDE model.

\begin{proposition} \label{prop:G_measurable}
The observation operator $G$ is $\mu^0$-measurable and satisfies
\begin{equation*}
\|G(u) \|_\Sigma \le C \left(\|u\|_{\ell^1} +  \|g\|_{H^{\frac{1}{2}}(\Gamma_N)} \right).
\end{equation*}
If the prior $\mu^0$ has $p$-th moment then $G \in L^p(\ell^1_\kappa; \mu^0)$.
\end{proposition}
\begin{proof}
Because $\mu^0$ is a measure on $(\ell^1_\kappa, \mathcal{F})$ and $\mathcal{F}$ is the Borel $\sigma$-Algebra we can apply Theorem \ref{cor:continuity_general_hh} for every point $z_j \in M_\kappa$ for $j = 1,...,m$ to show that $G : \ell^1_\kappa \rightarrow \mathbb{C}^m$ is continuous and in particular $\mu^0$-measurable. Again Theorem \ref{cor:solve_hh_general} applied to every observation point $z_j \in M_\kappa$ yields 
\begin{equation*}
\|G(u) \|_\Sigma \le C \left(\|u\|_{\ell^1} +  \|g\|_{H^{\frac{1}{2}}(\Gamma_N)} \right),
\end{equation*}
which shows the bound and $G \in L^p(\ell^1_\kappa; \mu^0)$ if $\mu^0$ has $p$-th moment.
\end{proof}

\begin{definition}
We define the potential $\Psi$ as follows:
\begin{equation}\label{Onepotential}
\Psi : \ell^1_\kappa \times \mathbb{C}^m \rightarrow \mathbb{R}, \quad \Psi(u, y) = \frac{1}{2} \|y - G(u) \|^2_\Sigma.
\end{equation}
\end{definition}

%

Given the regularity of the observation operator $G$ we are able to deduce several essential properties of the potential $\Psi$.
\begin{lemma} \label{lem:Potential}
Assume that $\mu^0$ has second moment. Then the potential $\Psi$ and the prior measure $\mu^0$ satisfy:
\begin{enumerate}
\item[(i)] $\Psi \ge 0$.
\item[(ii)] There exists a $\mu^0$-measurable set $X \subseteq \ell^1_\kappa$ and constants $K,C>0$ such that $\mu^0(X) > 0$ and
\begin{equation*}
\Psi(u, y) \le K + C \|y\|^2_\Sigma, \quad \text{for all } u \in X, \quad y \in \mathbb{C}^m.
\end{equation*}
\item[(iii)] For every $y \in \mathbb{C}^m$ the map $\Psi( \cdot, y) : \ell^1_\kappa \rightarrow \mathbb{R}$ is $\mu^0$-measurable. 
\item[(iv)] If $\mu^0$ has $p$-th moment for $p \ge 2$ then for every $y \in \mathbb{C}^m$ we have that $\Psi(\cdot, y) \in L^{\frac{1}{2}p}(\ell^1_\kappa; \mu^0)$.
\item[(v)] For every $r > 0$ there exists a $C > 0$ such that
\begin{equation*}
|\Psi(u, y_1) - \Psi(u, y_2)| \le C \left(\|u\|_{\ell^1} + 1\right) \| y_1 - y_2 \|_{\Sigma}
\end{equation*}
for all $u \in \ell^1_\kappa$ and for all $y_1, y_2 \in B_r(0)$.
\end{enumerate}
\end{lemma}
\begin{proof}
Property $(i)$ follows directly from the definition of the potential $\Psi$. To prove property $(ii)$, w.l.o.g. assume $\mathbb{E}_{\mu^0} [\|\cdot\|^2_{\ell^1}] > 0$ and apply Markov's inequality to $\|\cdot\|^2_{\ell^1}$ to obtain
\begin{align*}
\mu^0 (\{u \in \ell^1_\kappa : \|u\|^2_{\ell^1} \ge 2 \mathbb{E}_{\mu^0} [\|\cdot\|^2_{\ell^1}] \}) \leq \frac{\mathbb{E}_{\mu^0} [\|\cdot\|^2_{\ell^1}]}{2\mathbb{E}_{\mu^0} [\|\cdot\|^2_{\ell^1}]} = \frac{1}{2}.
\end{align*}
Thus we are able to define $X$ as the complementary set of events 
\begin{equation*}
X := \{u \in \ell^1_\kappa : \|u\|^2_{\ell^1} < 2 \mathbb{E}_{\mu^0} [\|\cdot\|^2_{\ell^1}] \}
\end{equation*}
and conclude $\mu^0(X) \ge \frac{1}{2}$. Using the triangle inequality, Young's inequality and Proposition \ref{prop:G_measurable}, we obtain for $u\in X$, $y\in\mathbb{C}^m$
\begin{equation*}
\Psi(u, y) \le \|y\|^2_\Sigma + \|G(u)\|^2_\Sigma \le \|y\|^2_\Sigma + C \|u\|_{\ell^1}^2 + C \leq \|y\|^2_\Sigma + C \mathbb{E}_{\mu^0} [\|\cdot\|^2_{\ell^1}] + C,
\end{equation*}		
which concludes $(ii)$. Property $(iii)$ follows because $G$ is $\mu^0$-measurable. Property $(iv)$ can be deduced from the definition of $\Psi$ because $G \in L^p(\ell^1_\kappa; \mu^0)$. Similarly $(v)$ is satisfied because
\begin{equation} \label{eq:proof:psiest}
\begin{split}
|\Psi(u, y_1) - \Psi(u, y_2)| & = \frac{1}{2}|(y_1 - G(u), y_1 - G(u))_\Sigma 
 - (y_1 - G(u), y_2 - G(u))_\Sigma \\
& \quad + (y_1 - G(u), y_2 - G(u))_\Sigma 
 - (y_2 - G(u), y_2 - G(u))_\Sigma| \\
& \leq \frac{1}{2} \|y_1 - y_2\|_\Sigma \|y_1 + y_2 - 2 G(u)\|_\Sigma,
\end{split}
\end{equation}

which proves the claim using Proposition \ref{prop:G_measurable} to estimate $\|G(u)\|_\Sigma$.
\end{proof}
\subsection{Posterior} \label{sec:posterior}
In the following, we define the posterior by means of a Radon-Nikodym derivative:
\begin{definition}
Let the observation $y \in \mathbb{C}^m$ be given. The posterior density with respect to the prior measure is defined by
\begin{equation} \label{PosteriorDef}
\frac{d\mu^y}{d\mu^0} := \frac{1}{\Lambda(y)} \exp(- \Psi(u, y)), \qquad \Lambda(y) := \int_{\ell^1_\kappa} \exp(- \Psi(u, y)) d\mu^0(u).
\end{equation}
\end{definition}
For more detailed discussion on the posterior measure, we refer to \cite[Section 2]{[Stu]}.
Next we prove that the posterior measure $\mu^y$ is well-defined.
\begin{theorem} \label{thm:WellDefinedPosterior}
Let $y\in\mathbb{C}^m$. Assume that $\mu^0$ has second moment. Then the posterior defined by \eqref{PosteriorDef} is a well-defined probability measure on $\ell^1_\kappa$. Moreover there exist constants $C_1, C_2 > 0$ independent of $y$ such that the normalization constant satisfies
\begin{align*}
 C_1 \exp(- C_2 \|y\|^2_\Sigma) \leq \Lambda(y) &\leq 1.
\end{align*}
\end{theorem}
\begin{proof}
We follow \cite[Theorem 4.1]{[Stu]}. First notice that the potential $\Psi(\cdot, y)$ is $\mu^0$-measurable. To prove the lower bound of $\Lambda(y)$ we apply Lemma \ref{lem:Potential} $(ii)$ to get an appropriate set $X \subseteq \ell^1_\kappa$ with $\mu^0(X) > 0$ such that
\begin{align*}
\int_{\ell^1_\kappa} \exp(- \Psi(u, y))d\mu^0(u) & \geq \int_{X} \exp(- \Psi(u, y)) d\mu^0(u) 
\geq \int_{X} \exp(- K - C \|y\|^2_\Sigma) d\mu^0(u) \\ 
& = \mu^0(X) \exp(-K) \exp(- C \|y\|^2_\Sigma) \geq C_1 \exp(- C_2 \|y\|^2_\Sigma).
\end{align*}
This implies $\Lambda(y) > 0$. For the upper bound we use Lemma \ref{lem:Potential} $(i)$ and that $\mu^0$ is a probability measure to immediately get 
\begin{equation*}
\Lambda(y) = \int_{\ell^1_\kappa} \exp(- \Psi(u, y)) d\mu^0(u) \le \int_{\ell^1_\kappa} 1 \, d\mu^0(u) = 1.
\end{equation*}
This proves well-definedness of the posterior.
\end{proof}
Using Gaussian noise we obtain exponential terms in the Radon-Nikodym derivatives. In order to estimate those, we require a simple lemma, which immediately follows from the mean value theorem.
\begin{lemma} \label{lem:AppendixExponentialHelper}
Let $a, b, c \ge 0$. Then the following bound holds: $\left|\exp\left(- a b \right) - \exp\left(- a c \right) \right| \leq a |b - c|$.
\end{lemma}
In the next theorem we show that the posterior measure is stable with respect to  small variations in $y$ in the Hellinger distance $d_{Hell}$. 
\begin{theorem}\label{thm:StabilityPosterior}
Assume that $\mu^0$ has second moment. Then for all $r > 0$ there exists $C > 0$ such that 
\begin{equation*}
d_{\text{Hell}}(\mu^{y_1}, \mu^{y_2}) \le C \|y_1 - y_2\|_\Sigma, \quad \text{for all } y_1, y_2 \in B_r(0).
\end{equation*}
\end{theorem}
\begin{proof}
First let us define the function $f(u, y) := \exp\left(- \frac{1}{2} \Psi(u, y)\right)$. This function satisfies
\begin{equation}\label{Short1}
\begin{split}
& \|f(\cdot, y_1) - f(\cdot, y_2)\|^2_{L^2(\ell^1_\kappa; \mu^0)} = \int_{\ell^1_\kappa} \left(\exp\left(- \frac{1}{2} \Psi(u, y_1)\right) - \exp\left(- \frac{1}{2} \Psi(u, y_2)\right)\right)^2  d \mu^0(u) \\
& \leq C \int_{\ell^1_\kappa} (\Psi(u, y_1) - \Psi(u, y_2))^2 d \mu^0(u) 
\leq C \int_{\ell^1_\kappa} (\|u\|_{\ell^1} + 1)^2 \|y_1 - y_2\|^2_\Sigma d \mu^0(u)
\leq C \|y_1 - y_2\|^2_\Sigma.
\end{split}
\end{equation}
Here we used Lemma \ref{lem:AppendixExponentialHelper} for the first inequality, Lemma \ref{lem:Potential} $(v)$ for the second inequality and that $\mu^0$ has second moment for the last inequality. The normalization constants as functions of $y$ are Lipschitz continuous. Applying the triangle inequality, Hölder's inequality and \eqref{Short1} shows this:
\begin{equation} \label{Short2}
\begin{split}
|\Lambda(y_1) - \Lambda(y_2)| & \leq \lVert f(\cdot, y_1) - f(\cdot, y_2) \rVert_{L^1( \ell^1_\kappa; \mu^0 )} \leq \lVert f(\cdot, y_1) - f(\cdot, y_2) \rVert_{L^2( \ell^1_\kappa; \mu^0 )} \leq C \lVert y_1 - y_2 \rVert_\Sigma.
\end{split}
\end{equation}


By Theorem \ref{thm:WellDefinedPosterior} we have
\begin{equation*}
\Lambda(y) \ge C_1 \exp(- C_2 \|y\|^2_\Sigma) \ge C > 0, \quad \text{for all } y \in B_r(0).
\end{equation*}
Since $x \mapsto x^{-\frac{1}{2}}$ is Lipschitz continuous on $[C,\infty)$ we find
\begin{equation}\label{Short3}
|\Lambda(y_1)^{-\frac{1}{2}} - \Lambda(y_2)^{-\frac{1}{2}}| \le C |\Lambda(y_1) - \Lambda(y_2)| \le C \|y_1 - y_2\|_\Sigma.
\end{equation}
With this result in mind take a look at two times the square of the Hellinger distance:
\begin{equation}\label{Short4}
\begin{split}
2 d_{\text{Hell}}(\mu^{y_1}, \mu^{y_2})^2 & = \int_{\ell^1_\kappa} \left( \Lambda(y_1)^{-\frac{1}{2}} \exp\left(- \frac{1}{2} \Psi(u, y_1)\right) - \Lambda(y_2)^{-\frac{1}{2}} \exp\left(- \frac{1}{2} \Psi(u, y_2)\right) \right)^2 d\mu^0(u) \\
& = \|\Lambda(y_1)^{-\frac{1}{2}} f(\cdot, y_1) - \Lambda(y_2)^{-\frac{1}{2}} f(\cdot, y_2) \|^2_{L^2(\ell^1_\kappa; \mu^0)} \\
& \le C \left(\Lambda(y_1)^{-\frac{1}{2}} \|f(\cdot, y_1) -  f(\cdot, y_2)\|_{L^2(\ell^1_\kappa; \mu^0)} + \|f(\cdot, y_2)\|_{L^2(\ell^1_\kappa; \mu^0)} |\Lambda(y_1)^{-\frac{1}{2}} - \Lambda(y_2)^{-\frac{1}{2}}| \right)^2.
\end{split}
\end{equation}
Now the desired estimate follows from the previous estimates together with $f \leq 1$ and $\Lambda(y) \geq C$. 
\end{proof}
Theorem \ref{thm:WellDefinedPosterior} has shown that the posterior is in fact well-defined. Often it is desirable to show existence of moments of it. 

\begin{theorem} \label{thm:MomentTheorem}
Assume that $\mu^0$ has $p$-th moment with $p \ge 2$ and let $(X, \|\cdot\|_X)$ be a Banach space. Furthermore, let $f : \ell^1_\kappa \rightarrow X$ be a $\mu^0$-measurable function such that for $\mu^0$-almost every $u \in \ell^1_\kappa$ there holds 
\begin{equation*}
\|f(u)\|_X \le C \|u\|^p_{\ell^1}.
\end{equation*}
Then for any $y\in\mathbb{C}^m$ there holds $f \in L^1(\ell^1_\kappa; \mu^y)$. Moreover, $\mu^y$ has $p$-th moment. 
\end{theorem}
\begin{proof}
There holds
\begin{align*}
\int_{\ell^1_\kappa} \|f(u)\|_X d \mu^y(u) 
& = \Lambda(y)^{-1} \int_{\ell^1_\kappa} \|f(u)\|_X \exp(- \Psi(u, y)) d \mu^0(u) \\ 
& \le C \Lambda(y)^{-1} \int_{\ell^1_\kappa} \|u\|^p_{\ell^1} \exp(- \Psi(u, y)) d \mu^0(u) 
\leq C \int_{\ell^1_\kappa} \|u\|^p_{\ell^1} d \mu^0(u),
\end{align*}
where we have used that $\Psi \ge 0$. This shows $f \in L^1(\ell^1_\kappa; \mu^y)$. 

Observe that the posterior $\mu^y$ has $p$-th moment since we can choose $f(u) = \|u\|^p_{\ell^1}$.
\end{proof}

\begin{theorem} \label{thm:expectationdistance}
Assume that $\mu^0$ has second moment and let $(X, \|\cdot \|_X)$ be a Banach space. Let $r>0$. Then there exists a $C>0$ such that the following holds: Let $y_1, y_2 \in B_r(0) \subseteq \mathbb{C}^m$. Let $f : \ell^1_\kappa \rightarrow X$ satisfy $f \in L^2(\ell^1_\kappa; \mu^{y_1})$ and $f \in L^2(\ell^1_\kappa; \mu^{y_2})$. Then there holds
\begin{equation*}
\|\mathbb{E}_{\mu^{y_1}}[f] - \mathbb{E}_{\mu^{y_2}}[f]\|_X \le C \|y_1 - y_2\|_\Sigma.
\end{equation*}
In particular, computing moments is stable under small perturbations in the measurements $y$.
\end{theorem}
\begin{proof}
\cite[Lemma 6.37]{[Stu]} delivers
\begin{align*}
\|\mathbb{E}_{\mu^1}[f] - \mathbb{E}_{\mu^2}[f]\|_X \le 2 \left(\mathbb{E}_{\mu^1}[\|f\|^2_X] + \mathbb{E}_{\mu^2}[\|f\|^2_X] \right)^{\frac{1}{2}} d_{\text{Hell}}(\mu^1, \mu^2).
\end{align*} 
Applying Theorem \ref{thm:StabilityPosterior} yields the desired result.
\end{proof}

\section{Sampling} \label{sec:sampling}

In this section, we derive a method to sample from the posterior. First, we introduce a discretized observation operator to obtain a computable discretized posterior measure and show several properties of it. We proceed to apply a Sequential Monte Carlo method to generate samples from the discretized posterior.

\subsection{Discrete Approximation of the Posterior Measure}
In practice we are not able to compute the solution of the Helmholtz equation exactly. As a consequence we have to replace the observation operator $G$ by its discrete approximation $G_h$ given by
\begin{equation*}
G_h(u) := (y_{u, h}(z_j))_{j = 1}^m.
\end{equation*}
The next lemma states several properties of the discrete observation operator and its relation to $G$.
\begin{lemma} \label{lem:DiscreteObservation}
There exist $h_0, C > 0$ such that for every $h \in (0, h_0]$ and every $u \in \ell^1_\kappa$ the discrete observation operator satisfies 
\begin{align}
\|G_h(u)\|_\Sigma &\le C \left(\|u\|_{\ell^1} + \|g\|_{H^{\frac{1}{2}}(\Gamma_N)} \right), \label{eq:boundedness_discrete_observation} \\
\|G(u) - G_h(u)\|_\Sigma &\le C |\ln h| h^2 \left(\|u\|_{\ell^1} + \|g\|_{H^{\frac{1}{2}}(\Gamma_N)} \right). \label{eq:error_estimate_discrete_observation}
\end{align}
Furthermore, $G_h$ is $\mu^0$-measurable and if $\mu^0$ has $p$-th moment then $G_h \in L^p(\ell^1_\kappa; \mu^0)$.
\end{lemma}
\begin{proof}
From Theorem \ref{thm:pointwise_disc_error_estimate} we have
\begin{equation*}
|y_u(z) - y_{u, h}(z)| \leq C |\ln h| h^2 \left( \|u\|_{\ell^1} + \lVert g \rVert_{H^\frac{1}{2}(\Gamma_N)} \right)
\end{equation*}
with a constant $C> 0$ independent of $u$ and $h$. This estimate applied to the measurement points $(z_j)_{j = 1}^m \in M_\kappa^m$ shows \eqref{eq:error_estimate_discrete_observation}. By the triangle inequality we have
\begin{equation*}
\|G_h(u)\|_\Sigma \leq \|G_h(u) - G(u)\|_\Sigma + \|G(u)\|_\Sigma,
\end{equation*}
which implies \eqref{eq:boundedness_discrete_observation} using the a-priori estimate, $h \leq h_0$ and Proposition \ref{prop:G_measurable}. $G_h$ is $\mu^0$-measurable because it is continuous by Corollary \ref{cor:discrete_continuity}. In particular,
\begin{align*}
\|G_h(u) - G_h(v)\|_\Sigma \leq C \displaystyle\sum\limits_{j = 1}^m \lvert y_{u, h}(z_j)-y_{v, h}(z_j)\rvert  \le C_\kappa m \left(
\|u\|_{\ell^1}+1+\vert \ln h \vert h^2
\right)\|u-v\|_{\ell^1}.
\end{align*} 
An application of \eqref{eq:boundedness_discrete_observation} under the assumption that $\mu^0$ has $p$-th moment shows $G_h \in L^p(\ell^1_\kappa; \mu^0)$.
\end{proof}
We further have to work with a discretized potential $\Psi_h$. Let us fix some arbitrary measurement observation $y \in \mathbb{C}^m$ in (\ref{eq:problem}) and  define the discrete potential 
\begin{equation*}
\Psi_h(u, y) := \frac{1}{2} \|y - G_h(u)\|^2_\Sigma. 
\end{equation*}
The next lemma states some essential properties of $\Psi_h$.
\begin{lemma} \label{lem:DiscretePotential}
There exist $C, h_0 > 0$ such that for every $h \in (0, h_0]$ the discrete potential satisfies for all $u\in \ell^1_\kappa, y\in \mathbb{C}^m$
\begin{align*}
|\Psi(u, y) - \Psi_h(u, y)| \le C |\ln h| h^2 \left( \|u\|^2_{\ell^1} + \|y\|^2_\Sigma + \lVert g \rVert^2_{H^\frac{1}{2}(\Gamma_N)} \right).
\end{align*}
Furthermore, Lemma \ref{lem:Potential} is valid for $\Psi_h$ instead of $\Psi$ with constants and sets independent of $h$.
\end{lemma} 
\begin{proof}
We compute similarly to \eqref{eq:proof:psiest} 
\begin{align*}
|\Psi(u, y) - \Psi_h(u, y)|
&\le C \|G(u) - G_h(u)\|_\Sigma (\|y\|_\Sigma + \|G(u)\|_\Sigma + \|G_h(u)\|_\Sigma).
\end{align*}
Now the a-priori error estimate in Lemma \ref{lem:DiscreteObservation} together with $\|G(u)\| \le C \|u\|_{\ell^1}$ and Young's inequality imply
\begin{align*}
|\Psi(u, y) - \Psi_h(u, y)| & \le C |\ln h| h^2 \left(\|u\|_{\ell^1} + \lVert g \rVert_{H^\frac{1}{2}(\Gamma_N)} \right) (\|y\|_\Sigma + \|u\|_{\ell^1}) \\
&\le C |\ln h| h^2 \left( \|u\|^2_{\ell^1} + \|y\|^2_\Sigma + \lVert g \rVert^2_{H^\frac{1}{2}(\Gamma_N)} \right).
\end{align*}
Lemma \ref{lem:Potential} holds with $\Psi$ replaced by $\Psi_h$ with constants and sets independent of $h$. This can be shown by a straight forward computation following the proof of Lemma \ref{lem:Potential}, which we omit for brevity.
\end{proof}

We apply Lemma \ref{lem:DiscreteObservation} and Lemma \ref{lem:DiscretePotential} to show that Lemma \ref{lem:Potential} holds if we replace $G$ and $\Psi$ by their discrete counterparts $G_h$ and $\Psi_h$. We emphasize that all the estimates in Lemma \ref{lem:Potential} are valid uniformly in $h$ as long as $h \in (0, h_0]$ for a suitably small $h_0$. Then applying Theorem \ref{thm:WellDefinedPosterior} ensures that the discrete posterior $\mu^y_h$ is well-defined, i.e. $\mu^y_h$ is defined by the Radon-Nikodym derivative with respect to the prior
\begin{equation}\label{RepresentativDensity}
\frac{d\mu^y_h}{d\mu^0} := \frac{1}{\Lambda_h(y)} \exp(- \Psi_h(u, y)), \qquad \Lambda_h(y) := \int_{\ell^1_\kappa} \exp(- \Psi_h(u, y)) d\mu^0(u).
\end{equation}

We summarize this result in the next theorem.
\begin{theorem} \label{thm:well_defined_discrete_posterior}
Let $y\in\mathbb{C}^m$. Assume that $\mu^0$ has a finite second moment. Then there exists $h_0 > 0$ such that for any $h \in (0, h_0]$ the discrete posterior measure $\mu^y_h$ is well-defined. There exist constants $C_1, C_2 > 0$ independent of $h$ such that the normalization constants satisfy
\begin{equation*}
C_1 \exp(- C_2 \|y\|^2_\Sigma) \leq \Lambda_h(y) \leq 1, \quad \text{for all } y \in \mathbb{C}^m.
\end{equation*}
In particular \eqref{RepresentativDensity} yields the well-definedness of $\mu_h^y$.
\end{theorem}
\begin{proof}
The proof is analogous to the proof of Theorem \ref{thm:WellDefinedPosterior} using the fact that Lemma \ref{lem:Potential} holds for the discrete potential $\Psi_h$ with constants and sets independent of $h$ by Lemma \ref{lem:DiscreteObservation}.
\end{proof}
The next theorem states that under certain conditions the rate of convergence in $h$ of the observation operator carries over to the rate of convergence of $\mu^y_h$ to $\mu^y$ in the Hellinger distance. 
\begin{theorem}\label{thm:HellingDistance2}
Let $y\in\mathbb{C}^m$. Assume that $\mu^0$ has a finite fourth moment. Then there exists a $h_0 > 0$ such that for any $h \in (0, h_0]$ we have 
\begin{equation*}
d_{\text{Hell}}(\mu^y, \mu^y_h) \le C |\ln h|h^2.
\end{equation*}
$C$ does not depend on $h$, but on $y$.
\end{theorem}

\begin{proof}
The proof of this theorem is similar to the proof of Theorem \ref{thm:StabilityPosterior}. The steps are in line with \cite[Theorem 4.9]{DashStuart2015}. However, let us mention few different steps compared to Theorem \ref{thm:StabilityPosterior}. Instead of $f(u,y_1), f(u,y_2)$ we have to consider $f(u) := \exp(-\frac{1}{2} \Psi(u,y))$ and $f_h(u) := \exp(-\frac{1}{2}\Psi_h(u,y)$. As in \eqref{Short1} we find
\begin{align*}
\|f(\cdot, y_1) - f(\cdot, y_2)\|^2_{L^2(\ell^1_\kappa; \mu^0)} & \leq C \int_{\ell^1_\kappa} (\Psi(u, y_1) - \Psi(u, y_2))^2 d \mu^0(u) \\
& \leq C \int_{\ell^1_\kappa} \left( \|u\|^2_{\ell^1} + \|y\|^2_\Sigma + \lVert g \rVert^2_{H^\frac{1}{2}(\Gamma_N)} \right)^2 d\mu^0(u) |\ln h|^2 h^4 \leq C |\ln h|^2 h^4.
\end{align*}
We used Lemma \ref{lem:DiscretePotential} in the second inequality and the assumption that $\mu^0$ has a finite fourth moment in the third inequality.

As in \eqref{Short2} and \eqref{Short3} we obtain the estimate
\begin{align*}
|\Lambda(y)^{-\frac{1}{2}} - \Lambda_h(y)^{-\frac{1}{2}}| \leq \lVert f(y) - f_h(y) \rVert_{L^2(\ell^1_\kappa; \mu^0)} \leq C |\ln h| h^2.
\end{align*}
Combining this and Theorem \ref{thm:well_defined_discrete_posterior} yields, with the analogous computations as in \eqref{Short4}, the desired estimate.

\end{proof}
\begin{theorem} \label{thm:f_h_convergence}
Let $y\in\mathbb{C}^m$. Assume that $\mu^0$ has a finite fourth moment. Let $(X, \|\cdot \|_X)$ be a Banach space and let $f : \ell^1_\kappa \rightarrow X$ have second moments with respect to  both $\mu^y$ and $\mu^y_h$. Then there holds
\begin{equation*}
\|\mathbb{E}_{\mu^y}[f] - \mathbb{E}_{\mu^y_h}[f]\|_X \le C |\ln h| h^2.
\end{equation*}
\end{theorem}
\begin{proof}
As Theorem \ref{thm:expectationdistance} this follows from \cite[Lemma 6.37]{[Stu]} using Theorem \ref{thm:HellingDistance2}.
\end{proof}

\subsection{Sequential Monte Carlo Method} \label{sec:SMC}
The discussions in Section \ref{sec:SMC} are limited to the continuous case for comprehension's sake. The definitions and statements can be straightforwardly generalized to the discrete situation using the previous results from Section \ref{sec:sampling}.

Throughout Section \ref{sec:SMC} let $y\in\mathbb{C}^m$ be a given observation.
In the following section we will use the Sequential Monte Carlo Method (SMC) from \cite{Dashti2017} to draw samples from the posterior measure. We also derive an error estimate. 

Let $J\in \mathbb{N}$ and for $j \in \{0,...,J\}$ define a sequence of measures $\mu_j$ that are absolutely continuous with respect to $\mu_0$ by
\begin{equation}\label{FurtherMuMeasure3}
\begin{aligned}
\frac{d\mu_j}{d\mu^0}(u) := \frac{1}{\Lambda_j} \exp\left(-j J^{-1} \Psi(u,y)\right), \qquad
\Lambda_j := \int\limits_{\ell^1_\kappa}\exp \left(-j J^{-1} \Psi(u,y) \right)d\mu^0(u).
\end{aligned}
\end{equation}
Note that $\mu_0$ is equal to the prior $\mu^0$ and $\mu_J$ equal to the posterior measure $\mu^y$. Our goal is to approximate $\mu_J$ sequentially using information of each $\mu_j$ to construct the next approximation $\mu_{j + 1}$. One idea behind the SMC is the approximation of each measure $\mu_j$ by a weighted sum of Dirac measures
\begin{align}\label{Dirac_particle_meas}
\begin{matrix}
\mu_j\approx \mu_j^N:=\displaystyle\sum_{n=1}^N w_j^{(n)}\delta_{u_j^{(n)}}
\end{matrix}
\end{align}
with $u_j^{(n)} \in \ell^1_\kappa$ and weights $w_j^{(n)} \geq 0$ that sum up to $1$. We define the sampling operator
\begin{align*}
S^N\colon  & \mathbb{P}(\ell^1_\kappa, \mathcal{F}) \rightarrow  \mathbb{P}(\ell^1_\kappa,\mathcal{F}), \\
& \nu \mapsto  \frac{1}{N} \sum_{n = 1}^N \delta_{u^{(n)}}, \quad u^{(n)} \sim \nu.
\end{align*}
Here $\mathbb{P}(\ell^1_\kappa, \mathcal{F})$ denotes the space of probability measures on $(\ell^1_\kappa, \mathcal{F})$. We also define
\begin{align*}
L\colon & \mathbb{P}(\ell^1_\kappa, \mathcal{F}) \rightarrow  \mathbb{P}(\ell^1_\kappa,\mathcal{F}), \\
& \nu  \mapsto \frac{\exp(- J^{-1}\Psi(u,y))}{\int_{\ell^1_\kappa} \exp(- J^{-1}\Psi(u,y))} \nu.
\end{align*}

We remark that the operator $L$ satisfies
\begin{equation*}
\mu_{j + 1} = L \mu_j, \quad j = 0,...,J - 1.
\end{equation*}

We further choose $P_j: \ell^1_\kappa \times \mathcal{F} \rightarrow [0,1]$ as a $\mu_j$-invariant Markov kernel. (For the concrete choice see Section \ref{sec:mcmc}.) That means
\begin{align*}
\mu_j(A) = \int_{\ell^1_\kappa} P_j(u, A)~d\mu_j(u) \qquad \text{for all } A\in\mathcal{F}.
\end{align*}
The idea of the kernel is to redraw samples in each iteration of the algorithm to better approximate $\mu_{j+1}$. (See Step 4. in Algorithm \ref{AlgoDasthiStuart}.) This allows us to define the discrete measures according to \cite[Algorithm 4]{Dashti2017} as follows 
\begin{align*}
\mu_0^N &:= S^N \mu^0, \\
\mu_{j + 1}^N &:= L S^N P_j \mu_j^N, \quad j = 0,...,J - 1. 
\end{align*}

\begin{algorithm}
	\begin{itemize}
		\item[1.] Let $\mu_0^N=S^N \mu_0$ and set $j=0$.
		\item[2.] Resample $u_j^{(n)}\sim \mu_j^N$, $n=1,\cdots,N$.
		\item[3.] Set $w_j^{(n)}=\frac{1}{N}$, $n=1,\cdots,N$ and define $\mu_j^N$ according to \eqref{Dirac_particle_meas}.
		\item[4.] Apply Markov kernel $\widehat{u}_{j+1}^{(n)}\sim P_j(u_j^{(n)},\cdot)$.
		\item[5.] Define $w_{j+1}^{(n)}=\widehat{w}_{j+1}^{(n)}/\left(\sum\limits_{\tilde{n}=1}^N \widehat{w}_{j+1}^{(\tilde{n})}, \right)$ with $\widehat{w}_{j+1}^{(n)}=\exp\left(-J^{-1}\Psi(\widehat{u}_{j+1}^{(n)},y)\right)w_j^{(n)}$ and $\mu^N_{j+1}:=\sum\limits_{n=1}^N \widehat{w}_{j+1}^{(n)} \delta_{\widehat{u}_{j+1}^{(n)}}$.
		\item[6.] $j \leftarrow j + 1$  and go to 2.
	\end{itemize}
	\caption{SMC-Algorithm}\label{AlgoDasthiStuart}
\end{algorithm}
In Algorithm \ref{AlgoDasthiStuart} we see the SMC as described in \cite[Section 5.3.]{Dashti2017}. 

\begin{remark}
Under the assumption that the potential satisfies
\begin{equation*}
\Psi^- \leq \Psi(u) \leq \Psi^+, \quad \text{for all } u \in \ell^1_\kappa 
\end{equation*}
for constants $\Psi^-, \Psi^+ \in \mathbb{R}$ convergence of this method is shown in \cite[Theorem 23]{Dashti2017}. In our setting we do not have an upper bound $\Psi^+$ for the potential and thus we prove a slight generalization.
\end{remark}

\begin{theorem} \label{thm:convergene_infinite_smc}
	For every measurable and bounded function $f$ the measure $\mu^N_J$ satisfies  
	\begin{equation} \label{eq:convergene_infinite_smc}
	\mathbb{E}_{\text{SMC}}[| \mathbb{E}_{\mu^N_J}[f] - \mathbb{E}_{\mu^y}[f] |^2] \leq \left(\sum_{j = 1}^J \left(2 \Lambda_J^{-1}\right)^j \right)^2 \frac{\|f\|^2_{\infty}}{N},
	\end{equation}
	where $\mathbb{E}_{\text{SMC}}$ is the expectation with respect to  the randomness in the SMC algorithm. 
\end{theorem}
\begin{proof} 
	We first prove a variant of \cite[Lemma 10]{Dashti2017} using similar techniques. Define $g := \exp( - J^{-1} \Psi)$ and let $f$ with $\|f\|_\infty \leq 1$ be given. Then from the proof of \cite[Lemma 10]{Dashti2017} and defining $\eta_j = S^N P_j \mu_j^N$ we conclude
\begin{align*}
& (L \mu_j)[f] - (L \eta_j)[f]  = \frac{1}{\mathbb{E}_{\mu_j}[g]}(\mathbb{E}_{\mu_j}[fg] - \mathbb{E}_{\eta_j}[fg]) + \frac{\mathbb{E}_{\eta_j}[fg]}{\mathbb{E}_{\eta_j}[g]} \frac{1}{\mathbb{E}_{\mu_j}[g]}(\mathbb{E}_{\eta_j}[g] - \mathbb{E}_{\mu_j}[g]).
\end{align*}
	A quick calculation shows
	\begin{equation*}
	\mathbb{E}_{\mu_j}[g] = \mathbb{E}_{\mu^0}[\exp( - jJ^{-1} \Psi)] \geq \mathbb{E}_{\mu^0}[\exp( - \Psi)] = \Lambda_J
	\end{equation*}
	and from $\|f\|_\infty \leq 1$ we conclude $|\mathbb{E}_{\eta_j}[fg] / \mathbb{E}_{\eta_j}[g]| \leq 1$. Hence we obtain
\begin{align} \label{eq:SMC_infinite_pre_result}
\begin{split}
|(L \mu_j)[f] - (L \eta_j)[f]| \leq \Lambda_J^{-1} |\mathbb{E}_{\mu_j}[fg] - \mathbb{E}_{\eta_j}[fg]| + \Lambda_J^{-1} |\mathbb{E}_{\eta_j}[g] - \mathbb{E}_{\mu_j}[g]|.
\end{split}
\end{align}
	We define the distance $d_{\text{op}}$ for probability measures $\nu, \eta$ on $(\ell^1_\kappa, \mathcal{F})$ as follows
	\begin{equation*}
	d_{\text{op}}(\nu, \eta) := \sup_{\|f\|_\infty \leq 1} \mathbb{E}_{\text{SMC}}[| \mathbb{E}_{\nu}[f] - \mathbb{E}_{\eta}[f] |^2]^{1 / 2}
	\end{equation*}
	and from $\|g\|_\infty \leq 1, \|f\|_\infty \leq 1$ together with \eqref{eq:SMC_infinite_pre_result} we conclude
	\begin{equation} \label{eq:SMC_L_result}
	d_{\text{op}}(L S^N P_j \mu_j^N, L \mu_j) \leq 2 \Lambda_J^{-1} d_{\text{op}}(S^N P_j \mu_j^N, \mu_j).
	\end{equation}
	Now we show that \cite[Theorem 23]{Dashti2017} holds in our setting. First we use \eqref{eq:SMC_L_result} and the triangle inequality to conclude 
	\begin{align*}
	d_{\text{op}}(\mu^N_{j + 1}, \mu_{j + 1}) &= d_{\text{op}}(LS^N P_j \mu_j^N, L \mu_j) \leq 2 \Lambda^{-1}_J d_{\text{op}}(S^N P_j \mu_j^N, \mu_j) \\
	&= 2 \Lambda^{-1}_J d_{\text{op}}(S^N P_j \mu_j^N, P_j \mu_j) \leq 2 \Lambda^{-1}_J (d_{\text{op}}(S^N P_j \mu_j^N, P_j \mu^N_j) + d_{\text{op}}(P_j \mu^N_j, P_j \mu_j)).
	\end{align*}
	A straight forward continuation similar to the proof of \cite[Theorem 23]{Dashti2017} leads to the statement
	\begin{equation} \label{eq:SMC_d_op_result}
	d_{\text{op}}(\mu^N_J, \mu^y) \leq \sum_{j = 1}^J \left(2 \Lambda^{-1}_J\right)^j \frac{1}{\sqrt{N}}.
	\end{equation}
	Here we remark that both \cite[Lemma 8]{Dashti2017} and \cite[Lemma 9]{Dashti2017} needed for the proof of \cite[Theorem 23]{Dashti2017} still hold. From \eqref{eq:SMC_d_op_result} we conclude that \eqref{eq:convergene_infinite_smc} holds for $\|f\|_\infty \leq 1$ and the statement for a general measurable bounded function $f$ follows by a scaling argument.
\end{proof}

We remark that it is possible to generalize this result if we replace the uniform tempering step $J^{-1}$ in \eqref{FurtherMuMeasure3} by non uniform steps. Therefore let $0 = \beta_0 < \beta_2 < ... < \beta_J = 1.$
We then define the tempered measures as follows
\begin{align*}
\mu_0 &= \mu^0, \\
\mu_{j + 1} &= L_j \mu_j = \Lambda^{-1}_j \exp((\beta_{j} - \beta_{j + 1}) \Psi(u, y)) \mu_j,
\end{align*}
with a suitable normalization constant $\Lambda_j$. The definition of $L$ has to be changed accordingly. (In fact one now requires multiple $L_j$'s.)

\subsection{Markov kernel} \label{sec:mcmc}
The SMC as presented in Algorithm \ref{AlgoDasthiStuart} requires us to choose $\mu_j$-invariant Markov kernels $P_j$. This choice is crucial for the performance of the SMC algorithm, see \cite[section 5.1.3]{cotter2013} or \cite[Remark 5.14]{[Stu]}. In this section we restrict ourselves to the temperature $j = J$, that is the posterior measure $\mu^y$, since results for other temperatures can be obtain by scaling the potential $\Psi$.

We use multiple steps of a standard Metropolis-Hastings algorithm to obtained a $\mu^y$-invariant Markov kernel $P_j$. Algorithm \ref{alg:mcmc} shows a single step of the Metropolis-Hastings algorithm found in \cite[Algorithm 1]{Dashti2017}. 
\begin{algorithm}
	\begin{itemize}
		\item[1.] Let $u \in \ell^1_\kappa$.
		\item[2.] Propose $u^\prime \sim q(u, \cdot)$.
		\item[3.] Draw $A$ from $\mathcal{U}[0, 1]$ independently of $(u, u^\prime)$.
		\item[4.] If $A \leq a(u, u^\prime)$ accept and set $v = u^\prime$.
		\item[5.] Otherwise reject and set $v = u$.
	\end{itemize}
	\caption{Metropolis-Hastings algorithm}\label{alg:mcmc}
\end{algorithm}
Under suitable conditions on the proposal distribution $q(u, \cdot)$ and for a particular choice of the acceptance probability $a$ it is possible to show $\mu^y$-invariance of Algorithm \ref{alg:mcmc}, that is if $u \sim \mu^y$ then $v \sim \mu^y$. We later prove such a statement in Theorem \ref{thm:mcmc} for a particular choice of $q$, for a general statement see \cite[Theorem 21]{Dashti2017}. It is well known that the efficiency of Metropolis-Hastings methods crucially depends on the choice of the proposal distribution $q$ and the acceptance probability $a$. Therefore we want to motivate the choices we make.

Every particle $u \sim \mu^y$ can be described in terms of the number of sources $k$, positions $x$ and amplitudes $\alpha$, hence for $u^\prime$ we propose $k^\prime$, $x^\prime$ and $\alpha^\prime$. Thus we examine proposals $u^\prime$ of the form
\begin{equation} \label{eq:proposal}
\begin{aligned}
k^\prime &= k, \\
x^\prime_\ell &= \begin{cases}
x_\ell + \gamma_x \eta_\ell, &\text{ if } x_\ell + \gamma_x \eta_\ell \in D_\kappa, \\
x_\ell, &\text{ otherwise},
\end{cases} \quad \ell \in \{1,...,k\}, \\
\alpha^\prime &= (1 - \gamma^2_\alpha)^{1 / 2}(\alpha - m_\alpha) + m_\alpha + \gamma_\alpha \xi.
\end{aligned}
\end{equation}
where $\gamma_x \geq 0$, $\gamma_\alpha \in [0, 1]$ and $m_\alpha \in \mathbb{C}^k$. We later explain our choice of $k^\prime = k$ in Example \ref{ex:noresamplingofnrofsources}.

Moreover, $\xi \sim N(0, \Gamma, C)$ is a multivariate complex Gaussian random variable with covariance matrix $\Gamma$ and relation matrix $C$ and $\eta_\ell \sim N(0, I)$ are standard multivariate Gaussians in $\mathbb{R}^d$ for any $\ell$. We assume independence of these random variables w.r.t. each other and $u$. The proposal chooses the identity proposal for the number of sources, a truncated Gaussian Random Walk for the positions such that they remain in $D_\kappa$ and a preconditioned Crank-Nicolson proposal for the amplitudes. The choice for such a proposal becomes apparent in the next theorem.
$\mathcal{U}(D^k_\kappa)$ denotes the uniform distribution in $D^k_\kappa$.
%
\begin{theorem} \label{thm:mcmc}
Let the prior given $k \in\mathbb{N}_0$ sources satisfy 
\begin{align*}
\mu^0(du | k) = \mu^0(d\alpha, dx | k) = \mu^0_{\alpha}(d\alpha | k) \mu^0_x(dx | k), \\
\mu^0_x(\cdot | k) = \mathcal{U}(D^k_\kappa), \quad \mu^0_\alpha(\cdot | k) = N(m_\alpha, \Gamma, C).
\end{align*}

Let $q(u,\cdot)$ be the proposal distribution associated with \eqref{eq:proposal} and define the acceptance probability as follows
\begin{equation*}
a(u, u^\prime) = \min\{1, \exp(\Psi(u) - \Psi(u^\prime))\}.
\end{equation*}
Then Algorithm \ref{alg:mcmc} is $\mu^y$-invariant.
\end{theorem}
\begin{proof}
The proof is given in the appendix.
\end{proof}

In the trivial case of $\Psi = 0$ such a proposal is always accepted and is therefore invariant with respect to the prior.

\begin{example} \label{ex:noresamplingofnrofsources}
Our numerical experiments show that the acceptance probability is very small if $k^\prime \not = k$. Such a behaviour is often found for dimension crossing MCMC, see \cite{RISy1997} or \cite{[Brooks]}. We expect such behaviour since the posterior measure given the number of sources $\mu^y(\cdot | k)$ may not be close to $\mu^y(\cdot | k^\prime)$ if $k^\prime \not = k$. We illustrate this behaviour in the next example. 

For the purpose of this example assume that the acceptance probability of the proposal which removes one source is given by
\begin{equation*}
a(u, u^\prime) = \min\{1, \exp(\Psi(u) - \Psi(u^\prime))\}.
\end{equation*}
Let us assume that we have a pair $(\alpha,x)\in \mathbb{C}\times D_\kappa$ such that $G(\alpha, x, 0, 0,...) \approx y / 3$. We define $u = (\alpha,x,\alpha,x,\alpha,x, 0, 0, \dots)$ and $u^\prime = (\alpha, x, \alpha, x, 0, 0, \dots)$ obtained by dropping the last two entries from $u$. A short computation using the properties of the observation operator now shows that
\begin{align*}
a(u, u^\prime) & = \exp(\|G(u) - y\|^2_\Sigma / 2) \exp(- \|G(u^\prime) - y\|^2_\Sigma / 2)  \\
& \approx \exp(0) \exp(- \lVert y/3 \rVert^2_\Sigma/2) = \exp(- \|y\|^2_\Sigma / 18),
\end{align*}
which may be very small. Such a result is visualized in Section \ref{sec:experiments} in Figure \ref{fig:P_emp_k}.
We conclude that a reasonable proposal for removing sources should modify the positions and amplitudes of the remaining source such that the data misfit is not too large. Such a proposal is non-trivial and out of the scope of this paper. The SMC algorithm \ref{AlgoDasthiStuart} ensures sufficient mixing properties through the resampling step, which allows the particles to change the number of sources $k$ to some other $k^\prime$. This of course requires that the prior with respect to the number of sources is chosen such that there are enough samples with $k^\prime$ sources. For the stated reasons we now only consider proposals where the number of sources is fixed.
\end{example}

\FloatBarrier
\section{Numerical Experiments} \label{sec:experiments}

\subsection{Sharpness of the proven estimates} \label{sec:setup_first}
In this section, we present numerical results for our prior model. Let $D = [0, 1]^2$ be the physical domain. We want to recover two sound sources placed in $x_{\text{exact},1}=(0.25,0.75)$ and $x_{\text{exact},2}=(0.75,0.75)$ with amplitudes $\alpha_{\text{exact},1}=10+10i$ and $\alpha_{\text{exact}, 2}=10+10i$. We also choose $g=0$. The three measurement points $(z_\ell)_{\ell=1}^3$ are located at the following positions
\begin{align*}
\begin{matrix}
z_1=(0.1,	0.5), &
z_2=(0.5,	0.5), &
z_3=(0.9,	0.5).
\end{matrix}
\end{align*}
We choose $\kappa = 0.05$ and thus as source domain $D_{\kappa} = [0.1,0.9]\times [0.6,0.9]$. We indeed have $z_1,\dots, z_3\in M_\kappa$.

For the parameters of the Helmholtz equation we consider the fluid density $\rho=1$, frequency $\zeta=30$, sound speed $c=5$, and coefficients $\alpha(\zeta)=1$, $\beta(\zeta)= 1 / 30$ for the isolating material on the boundary $\Gamma_Z = \Gamma$. 

We define 
\begin{align*}
u_{\text{exact}} = (\alpha_{\text{exact},1}, x_{\text{exact},1}, \alpha_{\text{exact},2}, x_{\text{exact},2}, 0,0,\dots )\in \ell^1_\kappa
\end{align*}
and $y_{\text{exact}} := G_h(u_{\text{exact}})$. We chose the prior number of sources $k$ to be Poisson distributed with expectation $2$, that is $k \sim \operatorname{Poi}(2)$. Given the number of sources $k$, we choose the amplitudes $\alpha_1^k,\dots, \alpha_k^k$ to be i.i.d. $N(10 + 10 i, 2, 0)$.

The prior source positions are i.i.d. uniformly distributed $x_1^k, \dots, x_k^k \sim \mathcal{U}(D_\kappa)$. We in particular assume that the random variables $\alpha_\ell^k$ and $x^k_\ell$ are independent. For the observational noise we consider $\eta = (\eta_\ell)_{\ell = 1}^3$ with $\eta_1$, $\eta_2$, $\eta_3$ all i.i.d. $N(0, 0.1, 0)$.  We further define $\xi = (\xi_1,...,\xi_k)$ with $\xi_\ell \sim N(0, 2, 0)$ i.i.d. random variables and set $\gamma_\alpha = 0.4$ and $\gamma_x = 0.1$ for the proposal in \eqref{eq:proposal}. For the applicability of Theorem \ref{thm:mcmc} we have $m_\alpha = 10+10i$. We use non uniform tempering steps $\beta = (0, 0.03, 0.3, 1)^T$ and we apply the $\mu_j$-invariant Markov kernel from Algorithm \ref{alg:mcmc} ten times to every particle in each tempering step. The parameters $\gamma_\alpha, \gamma_x$ for the random proposals are chosen such that the average acceptance probability of the Markov kernel $P_2$ is approximately $0.30$.

We triangulate the domain $D$ uniformly with triangles of diameter $h$.

\subsubsection{Source Separation Capabilities} \label{sec:ssc}

\begin{figure*}
\includegraphics[width=\textwidth]{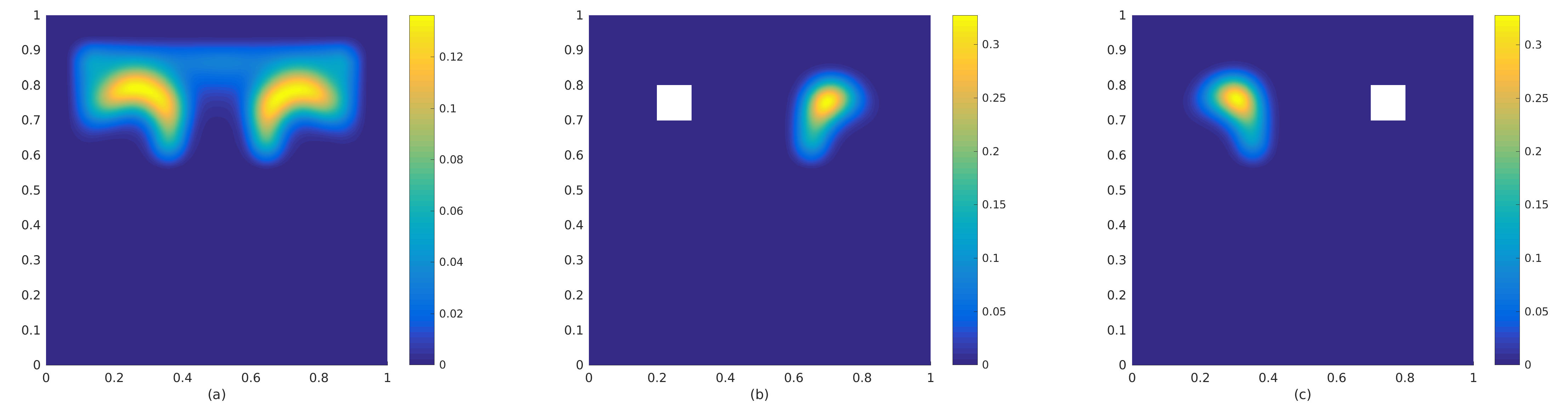}
\caption{(a) shows the function $P_{\text{emp}}$ in $[0, 1]^2$. (b) shows $P_{\text{emp}}( \cdot |Q, k = 2)$ if one of the sources is in the white quadrant $Q = [0.2, 0.3] \times [0.7, 0.8]$ and (c) for the quadrant $Q = [0.7, 0.8] \times [0.7, 0.8]$.}
\label{fig:Experiment1}
\end{figure*}

\begin{figure*}
\includegraphics[width=\textwidth]{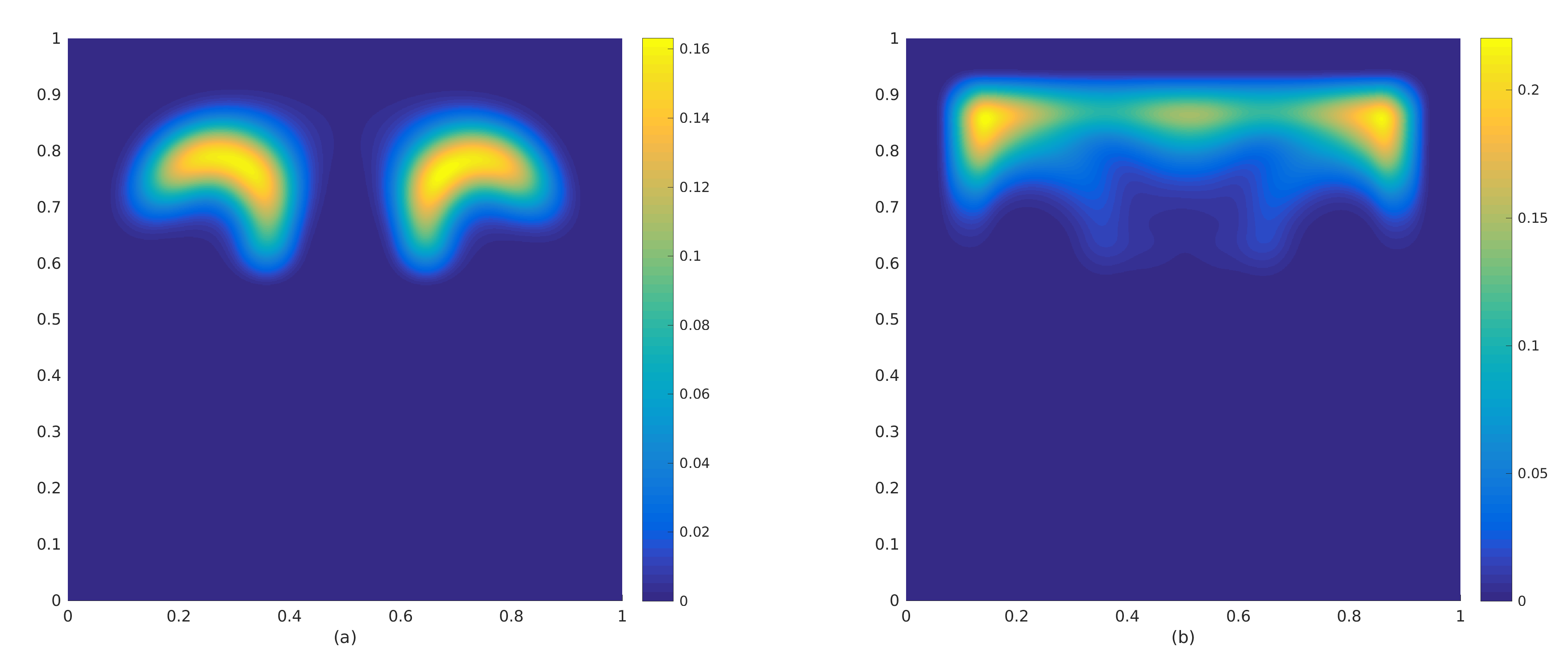}
\caption{(a) shows the function $P_{\text{emp}}(\cdot | k = 2)$ in $[0, 1]^2$ and (b) shows $P_{\text{emp}}(\cdot | k = 3)$. We conclude that $\mu^y(\cdot | k = 2)$ differs significantly from $\mu^y(\cdot | k = 3)$. C.f. the arguments in Example \ref{ex:noresamplingofnrofsources}.}
\label{fig:P_emp_k}
\end{figure*}

We now present results on the source positions of samples from the posterior. We consider the SMC as introduced in Section \ref{sec:sampling} with $N = 10^7$ samples. We choose a uniform mesh of size $h = \sqrt{2} \cdot 2^{-7}$ for all experiments Section \ref{sec:ssc}.

For $n=1,\dots,N$ Algorithm \ref{AlgoDasthiStuart} delivers the particles 
$u^{(n)} = (\alpha_{1}^{(n)},x_{1}^{(n)}, \dots, \alpha_{k^{(n)}}^{(n)}, x_{k^{(n)}}^{(n)},0,0,\dots) \in \ell^1_\kappa$ with weights $w^{(1)},\dots, w^{(N)}$. The function $P_{\text{emp}}$ is defined as approximation to the probability that a source is located at $x \in D$ in the following sense
\begin{align*}
P_{\text{emp}}(x) & := \sum_{n=1}^N w^{(n)} \max_{1 \leq \ell \leq k^{(n)}} \left\{ K_\varepsilon \left(x - x^{(n)}_{\ell}\right) \right\} \\
& \approx \mathbb{P}_{\mu^y}(u \text{ has a source in } B_\varepsilon(x)).
\end{align*}
The function $K_\varepsilon$ is a smooth cut-off function approximating the indicator function $1_{B_\varepsilon(0)}$ for $\varepsilon = 0.04$ defined as
\begin{equation*}
K_\varepsilon(x) := 
\begin{cases}
1, &\text{if } \|x\| \leq \varepsilon, \\
\frac{1}{2} + \frac{1}{2} \cos\left(\frac{2}{\varepsilon} \pi (\|x\| - \varepsilon) \right), &\text{if } \varepsilon \leq \|x\| \leq \frac{3}{2} \varepsilon, \\ 
0, &\text{otherwise.}
\end{cases}
\end{equation*}
In Figure \ref{fig:Experiment1} (a) we can see the function $P_{\text{emp}}$ recovers the true source positions $x_{\text{exact},1}$ and $x_{\text{exact},2}$ quite well. 

We further want to analyse positions of pairs. In particular we would like to see if for a sample with two sources the positions are close to the $x_{\text{exact},1}$ and $x_{\text{exact}, 2}$. Given that a sample has two sources, i.e. $k=2$ and one is located in $Q \subseteq D_\kappa$, we look for the probability that the second source is close to some $x\in D$ but not in $Q$. Formally for the index set $I(Q, k = 2) = \{n \in \{1,...,N\} | x_{1}^{(n)} \in Q \text{ or } x_{2}^{(n)} \in Q, \, k^{(n)} = 2 \}$ we define
\begin{align*}
P_{\text{emp}}(x | Q, k = 2) & := \frac{1}{\sum_{n \in I(Q, k = 2)} w^{(n)}} \sum_{n \in I(Q, k = 2)} w^{(n)} \max_{ x^{(n)}_{\ell} \not \in Q, \ell=1,2 } K_\varepsilon \left(x - x^{(n)}_\ell\right) \\
& \approx \mathbb{P}_{\mu^y}(u \text{ has a source in } B_\varepsilon(x) \setminus Q \, | \, u \text{ has a source in } Q, \, k = 2).
\end{align*}
Figure \ref{fig:Experiment1} (b) and (c) show results for different $Q$. We conclude that if one of the sources is located near $x_{\text{exact}, 1}$ then the other one is near $x_{\text{exact}, 2}$ and vice versa. This shows that Algorithm \ref{AlgoDasthiStuart} correctly identifies pairs of sound sources.

\subsubsection{Convergence in Mean Square Error for Functions}
Theorem \ref{thm:convergene_infinite_smc} states linear convergence in Mean Square Error (MSE) for bounded measurable functions $f$ such that

\begin{equation*}
\mathbb{E}_{\text{SMC}}[| \mathbb{E}_{\mu^N_J}[f] - \mathbb{E}_{\mu^y}[f] |^2] \leq C \frac{\|f\|^2_{\infty}}{N}.
\end{equation*}
We are interested in testing the convergence of the functions $f_1,...,f_5$ listed in Table \ref{table:MSE_functions} to extract information from the posterior measure.
\begin{table*}[t] 
\centering 
\begin{tabular}{ | c | c | } 
\hline 
$f$ & Description for $\mathbb{E}_{\mu^y}[f]$ \\ \hline
$f_1(u) := \|u\|_{\ell^1}$ & First moment \\ \hline  
$f_2(u) := 1_{\{u \text{ has exactly two sources}\}}(u)$ & Probability to have exactly two sources \\ \hline
$f_3(u) := |y_u(z_{\operatorname{prediction}})|$ & Expected pressure amplitude at $z_{\operatorname{prediction}}$ \\ \hline
$f_4(u) := \operatorname{Var}(|y_u(z_{\operatorname{prediction}})|)$ & Variance of pressure amplitude at $z_{\operatorname{prediction}}$ \\ \hline 
$f_5(u) := 10 \log_{10}(\max(1,|\operatorname{Re}(y_u(z_{\operatorname{prediction}}) \exp(-i \zeta t))|))$ & Decibel function at $z_{\operatorname{prediction}}$ \\ \hline
\end{tabular}
\caption{List of functions used in the experiments in Figure \ref{fig:MSE} and Figure \ref{fig:f_h_convergence}.}
\label{table:MSE_functions} 
\end{table*}

We choose $z_{\operatorname{prediction}} := (1 / 2, 1 / 4) \in D$ to predict the pressure at a point distinct from the measurement points $z_1,...,z_3$. For $f_5$ we choose the time $t = 1$. We remark that Theorem \ref{thm:convergene_infinite_smc} shows convergence results only for $f_2$. Nevertheless we observe convergence for the other functions in Figure \ref{fig:MSE} as well. For this experiment we fix the mesh size to $h_{\operatorname{ref}} = \sqrt{2} \cdot 2^{-7}$ and only vary the number of samples. The reference measure is computed using $N_{\operatorname{ref}}= 10^7$ samples and the other measures use $N = 100 \cdot 2^k$ samples with $k = 1,...,9$. For the expectation integral "$\mathbb{E}_{\operatorname{SMC}}$" we average the result over $100$ SMC runs.

\begin{figure*}[t!] 
\begin{center}
\includegraphics[width=0.5\textwidth]{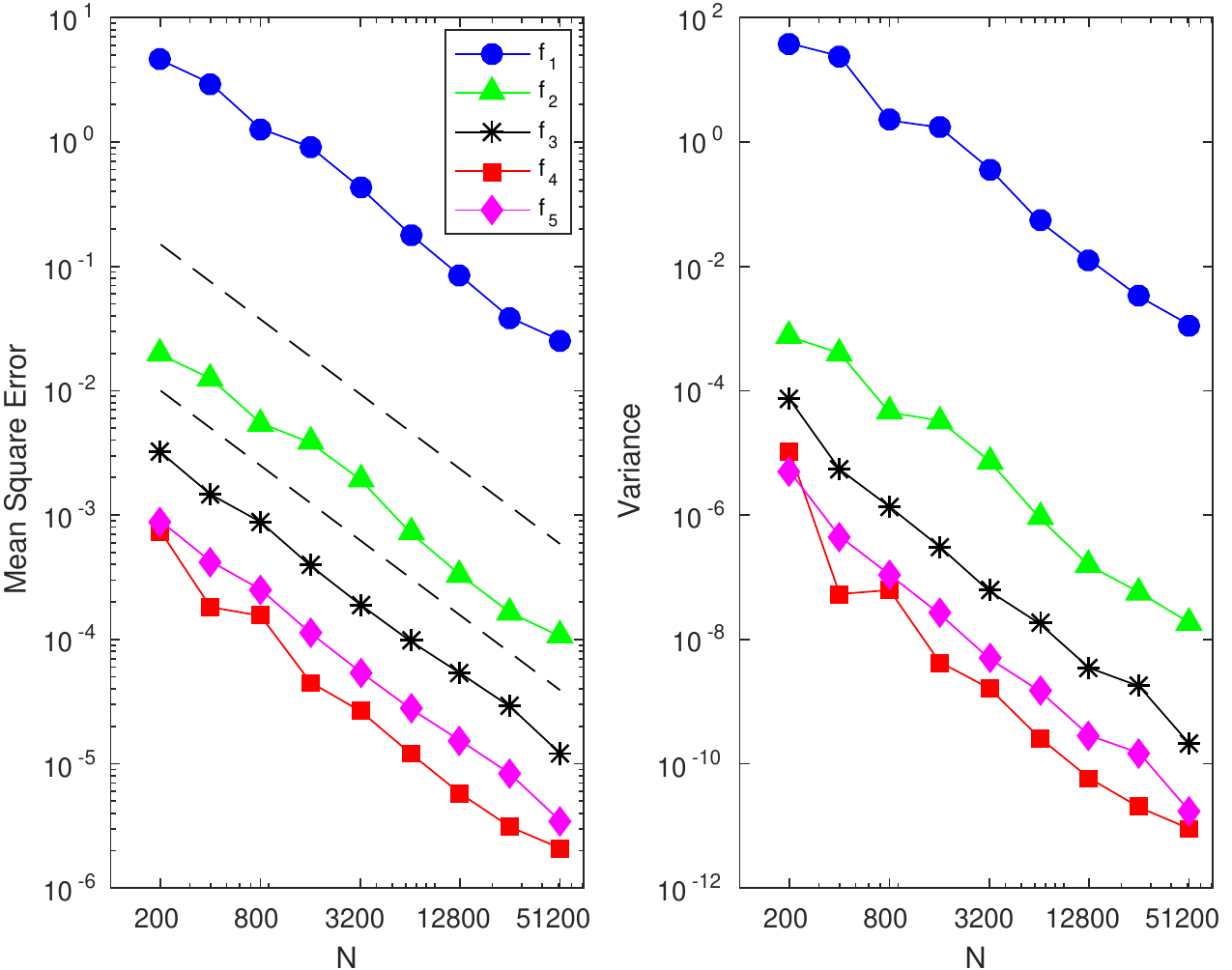}
\end{center}
\caption{On the left figure we see the MSE for the functions in Table \ref{table:MSE_functions}. The rate of $\mathcal{O}(1 / N)$ (drawn dashed) is proven only for $f_2$. In the right figure we see the variance of the MSE averaged over $100$ runs for each $N$ on the x-axis.} \label{fig:MSE}
\end{figure*}

\subsubsection{Convergence in the Hellinger Distance} \label{section:HellMSESection}
We verify Theorem \ref{thm:HellingDistance2} numerically. It states convergence of the discretized to the true posterior with respect to the mesh parameter $h$. More specifically  
\begin{equation*}
d_{\text{Hell}}(\mu^y_h, \mu^y) \leq C |\ln h| h^2.
\end{equation*}
The definition of the Hellinger Distance requires us to evaluate the Radon-Nikodym derivatives $d\mu^y_h / d\mu^0$ and $d \mu^y / d\mu^0$ for the same samples from the prior. Instead we use the equivalent formulation (see Appendix \ref{sec:appendix_hellinger})
\begin{equation} \label{eq:hellinger_numerics}
\begin{aligned}
4 d^2_{\text{Hell}}(\mu^y_h, \mu^y) &= \int_{\ell^1_\kappa} \left(1 - \left(\frac{d\mu^y_h}{d \mu^y}(u) \right)^{1 / 2} \right)^2 d\mu^y(u) \\
								&+ \int_{\ell^1_\kappa} \left(1 - \left(\frac{d\mu^y}{d \mu^y_h}(u) \right)^{1 / 2} \right)^2 d\mu^y_h(u),
\end{aligned}
\end{equation}
which allows us to approximate the integrals using samples from the SMC method. In this experiment we consider a fine mesh with $h_{\operatorname{ref}} := \sqrt{2} \cdot 2^{-7}$ and approximate $\mu^y \approx \mu^y_{h_{\operatorname{ref}}}$. We compare this measure with $\mu^y_h$ for $h =\sqrt{2} \cdot 2^{-k}$ for $k=2,...,6$. We fix $N = 5 \cdot 10^5$ and averaged the approximated Hellinger Distance over $50$ runs. Figure \ref{fig:HellingerDist} shows that the convergence rate $\mathcal{O}(|\ln h|h^2)$ holds sharply.
\begin{figure*}
\centering
\includegraphics[width=0.5\textwidth]{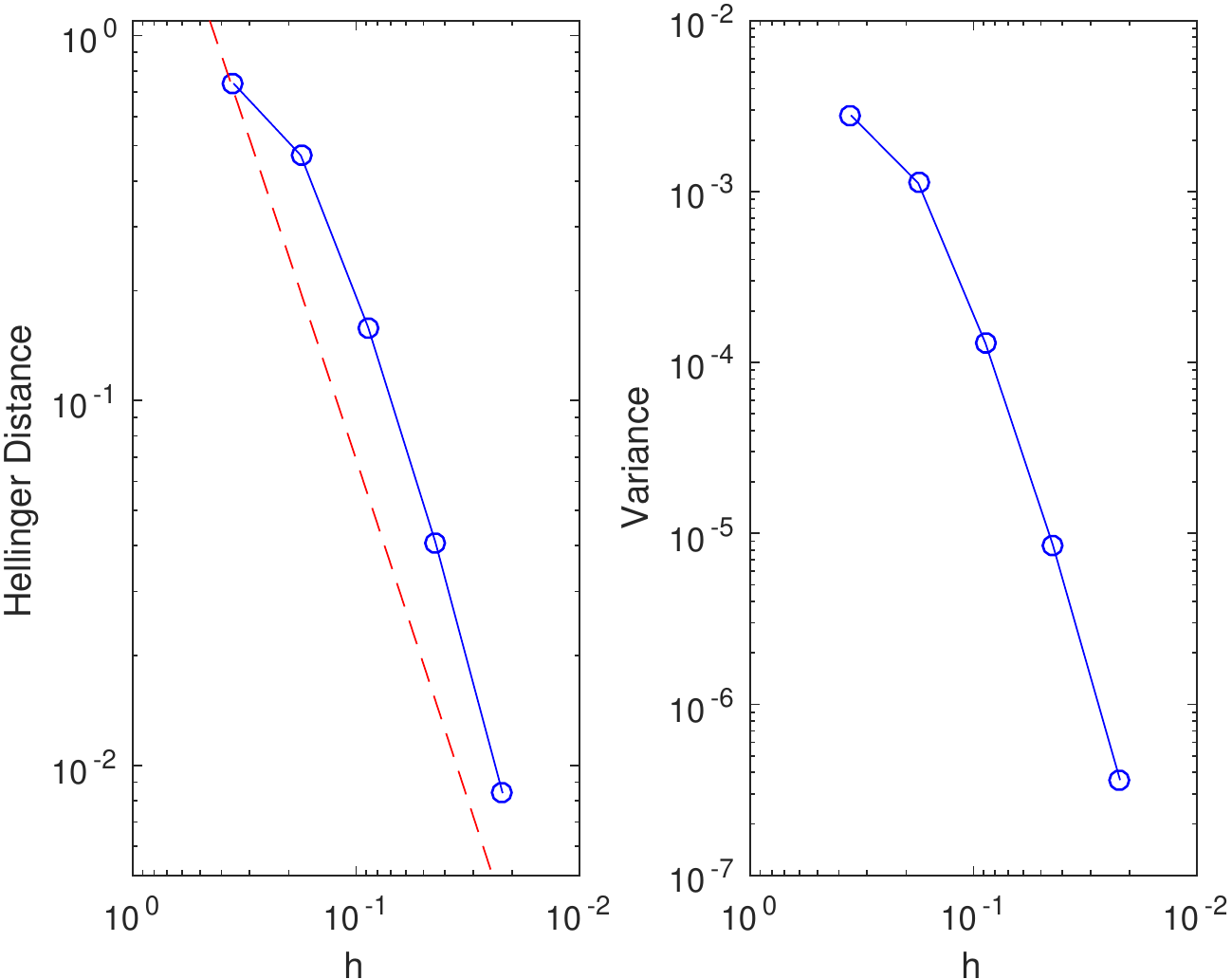}
\caption{The curve of the left figure shows the Hellinger distance between the discretized measure $\mu^y_h$ and the approximation of $\mu^y$ averaged over $50$ runs. The dashed line shows the reference $\mathcal{O}(|\ln h|h^2)$. The right figure shows the variance of the Hellinger Distance for different $h$ computed from $50$ runs.}\label{fig:HellingerDist}
\end{figure*} 
\subsubsection{Convergence in \texorpdfstring{$\mathbf{h}$}{Lg} for functions} 
Finally, we want to verify Theorem \ref{thm:f_h_convergence} which states the convergence of the error
\begin{equation} \label{eq:f_h_convergence_bound}
e_h(f) := \|\mathbb{E}_{\mu^y}[f] - \mathbb{E}_{\mu^y_h}[f]\|_X \le C |\ln h| h^2
\end{equation}
for functions $f$ with second moments with respect to both $\mu^y$ and $\mu^y_h$. We use the functions $f_1,...,f_5$ defined in Table \ref{table:MSE_functions}. For all experiments we choose $N_{\operatorname{ref}} = 5 \cdot 10^5$ samples from the SMC sampler. We approximate $\mu^y$ using $h_{\operatorname{ref}} = \sqrt{2} \cdot 2^{-7}$. For the other measures we vary the mesh parameter $h = \sqrt{2} \cdot 2^{-k}$ for $k = 2,...,6$ and average the values of $\mathbb{E}_{\mu^y_h}[f]$ and $\mathbb{E}_{\mu^y}[f]$ over $50$ runs. Figure \ref{fig:f_h_convergence} shows the values of $e_h$ and indicates that the proven convergence rates are sharp.

\begin{figure*}
\centering
\includegraphics[width=0.5\textwidth]{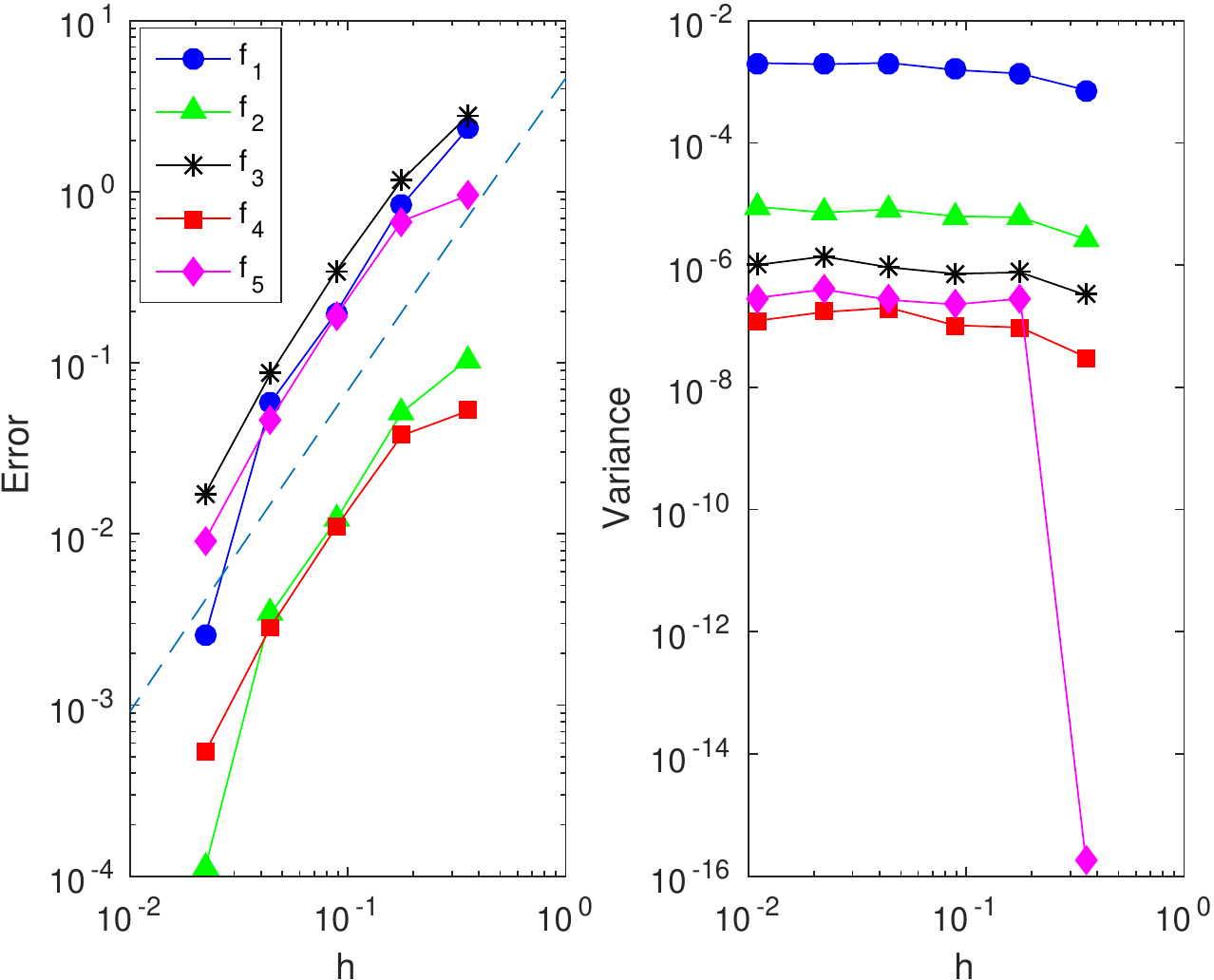}
\caption{The left figure shows the error $e_h(f_i)$ from \eqref{eq:f_h_convergence_bound} for the functions in Table \ref{table:MSE_functions}. The dashed line is the reference convergence rate $\mathcal{O}(|\ln h |h^2)$. The right figure shows the variance of approximations to $\mathbb{E}_{\mu^y_h}[f_i]$ for $h = \sqrt{2} \cdot 2^{-k}$ with $k = 2,...,7$ computed using $50$ runs. The value for $k = 7$ corresponds to the reference measure.}\label{fig:f_h_convergence}
\end{figure*} 


\subsection{An example with more sources and the MAP estimator} \label{sec:Experiment2}

In this section we want to demonstrate the feasibility of our method for more than two sources. Unless mentioned otherwise, in this paragraph we choose the same parameters as described in section \ref{sec:setup_first}. For this experiment we assume 5 sources $x_1,...,x_5$ and 9 measurement locations $z_1,...,z_{9}$, see Figure \ref{fig:Experiment2}. We change the prior location of the sources to be $\mathcal{U}(D_\kappa)$ with $D_\kappa$ being $U$ shaped as follows 
\begin{align*}
D_\kappa &:= D_1 \cup D_2 \cup D_3, \\
D_1 &:= [0.1, 0.9] \times [0.1, 0.6], \quad D_2 := [0.1, 0.35] \times [0.6, 0.9], \quad D_3 := [0.65, 0.9] \times [0.6, 0.9].
\end{align*} 
The experiment is depicted in figure \ref{fig:Experiment2}. The prior for the number of sources is distributed as $k \sim \operatorname{Poi}(4)$ and the observational noise $\eta_1,...,\eta_{9} \sim N(0, 0.1, 0)$ i.i.d.. For the Markov Kernel we choose $\gamma_\alpha = 0.4$ and $\gamma_x = 0.05$ such that the acceptance ratio is roughly $0.25$. We use $N = 10^7$ samples for the method and fix the mesh size $h = \sqrt{2} \cdot 2^{-7}$. 

Next we address the notion of a maximum a posteriori estimator (MAP). We define the empirical MAP-index and the empirical MAP-index given $k$ sources as follows
\begin{equation}\label{MAPMAP}
\begin{split}
n_{\operatorname{MAP}}:=\argmax\limits_{
	\begin{matrix}
	n \in \{1,...,N\}
	\end{matrix}} w^{(n)}, \qquad
n_{\operatorname{MAP}}^k:=\argmax\limits_{
	\begin{matrix}
	n \in \{1,...,N\},\\
	k^{(n)} = k
	\end{matrix}} w^{(n)}.
\end{split}
\end{equation}
In our experiments both quantities are unique. 

\begin{figure*}
\centering
\includegraphics[width=0.5\textwidth]{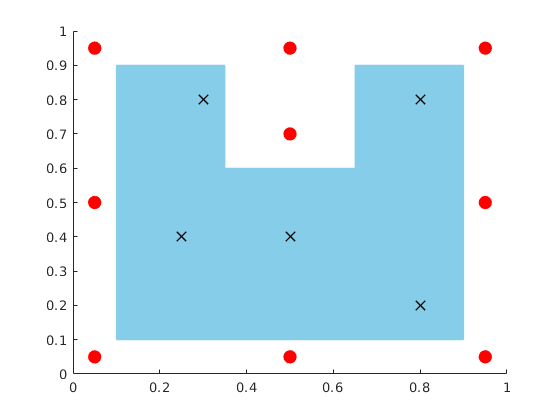}
\caption{ The setup for the experiment described in Section \ref{sec:Experiment2}. The crosses denote the position of the sources and circles the measurement positions. The area where the prior can have sources is u-shaped.}\label{fig:Experiment2}
\end{figure*} 


Some notable quantities are depicted in table \ref{tab:results_experiment_2}. We neglect samples with $k < 3$ and $k > 7$, since their probability is negligible. For this experiment the MAP estimator has the correct number of sources, but the weight $w^{(n_{MAP}^5)}$ is close to $w^{(n_{MAP}^6)}$. Notice that the prior prefers $5$ sources over $6$ sources, which shows that the MAP estimator alone does not seem to be appropriate to infer the number of sources. For this experiment the uncertainty in the number of sources gets reduced significantly. Most notably, the posterior probability for $k = 3$ is less than one percent. 



\begin{table*}[t]
  \centering
  \begin{tabular}{|M{1.1cm}|M{1.1cm}|M{2.2cm}|M{2.2cm}|}
    \hline
    No. of Sound Sources $k$ & $\mathbb{P}_{\mu^0}(k)$ & $\mathbb{P}_{\mu_J^N}(k)$ & $w^{(n_{MAP}^k)}$\\ \hline
    3  & 0.195	& 0.004 & $7.53 \cdot 10^{-7}$ \\ \hline
    4  & 0.195  & 0.281 & $1.09 \cdot 10^{-5}$ \\ \hline
    5  & 0.156	& 0.543 & $1.50 \cdot 10^{-5}$ \\ \hline
    6 & 0.104	& 0.156 & $1.48 \cdot 10^{-5}$ \\ \hline
    7 & 0.060	& 0.014 & $9.84 \cdot 10^{-6}$ \\ \hline
  \end{tabular}
  \caption{ The table shows the prior probability density $\mathbb{P}_{\mu^0}(k)$ with $k \sim \operatorname{Poi}(4)$, the posterior probability for $k$ sources $\mathbb{P}_{\mu^N_J}(k)$ and the weights for the MAP estimator.}\label{tab:results_experiment_2}
\end{table*}

\begin{figure*}
\centering
\includegraphics[width=0.3\textwidth]{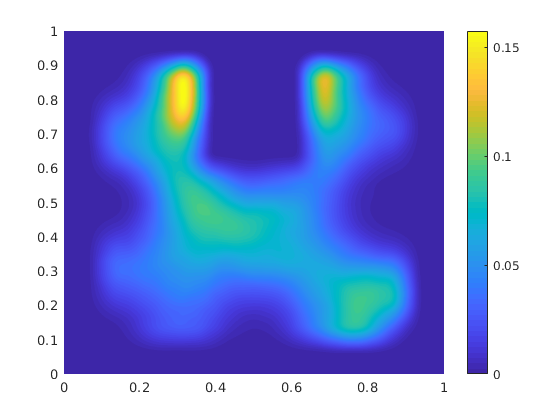}
\includegraphics[width=0.3\textwidth]{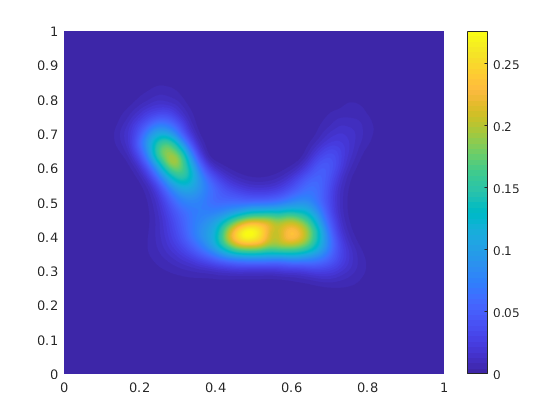}
\includegraphics[width=0.3\textwidth]{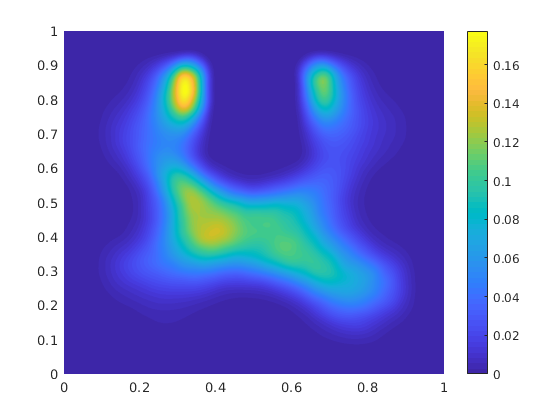}
\includegraphics[width=0.3\textwidth]{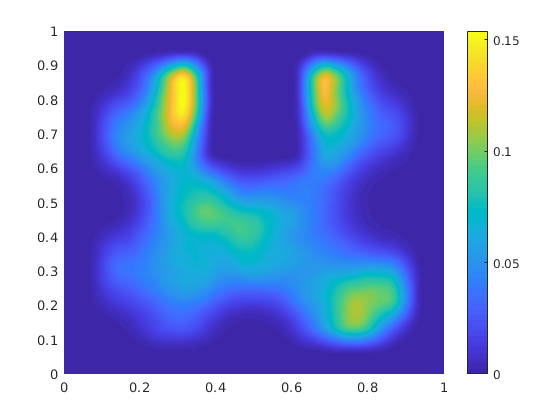}
\includegraphics[width=0.3\textwidth]{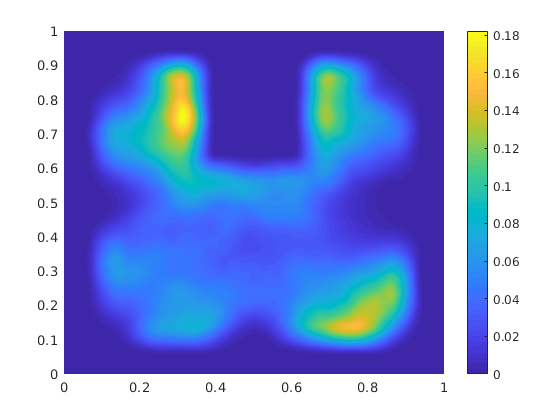}
\includegraphics[width=0.3\textwidth]{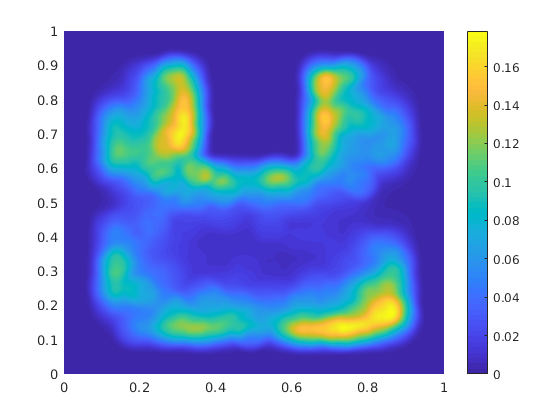}
\caption{The top left image shows $P_{\operatorname{emp}}(\cdot)$, the top middle $P_{\operatorname{emp}}(\cdot | k = 3)$ and the top right $P_{\operatorname{emp}}(\cdot | k = 4)$. The bottom row shows $P_{\operatorname{emp}}(\cdot | k)$ for $k = 5, 6, 7$ from left to right.}
\label{fig:p_emp_experiment_2}
\end{figure*}

\section*{Acknowledgements}
All authors gratefully  acknowledge  support  from  the  International  Research Training Group IGDK 1754 „Optimization and Numerical Analysis for Partial Differential Equations with Nonsmooth Structures“, funded by the German Research Foundation (DFG) and the Austrian Science Fund (FWF): [W 1244-N18].

Dominik Hafemeyer acknowledges support from the graduate program TopMath of the Elite Network of Bavaria and the TopMath Graduate Center of TUM Graduate School at Technische Universität München. He is also a scholarship student of the Studienstiftung des deutschen Volkes.

We thank Elisabeth Ullmann for the support she gave us by proof reading the draft of this paper and her helpful hints to improve the paper. We further thank Jonas Latz for supplying us with useful background on the Sequential Monte Carlo method.

\begin{appendix} 

\section{Appendix} 

\subsubsection*{Proof of Proposition \ref{lem:pointwise_green_disc_error}}

The proofs in this part of the appendix follow the proofs in \cite[Section 4]{[BeGaRo]}. The authors explicitly state though that they to not analyse the dependence of the appearing constants with respect to $\|x-z\|$, where $x \in D_\kappa$ denotes the position of a Dirac and $z \in M_\kappa$ an arbitrary measurement point. First we restate a slight variant of \cite[Corollary 5.1]{SchatzWahlbinMaxNormEst}. In the following, let us consider the finite element space
\begin{align*}
\widetilde V_h := \left\lbrace v_h\in H^1(D,\mathbb{R})\cap C(\bar D,\mathbb{R}): v_h|_T \text{ is affine linear } \forall T\in \mathcal{T}_h \right\rbrace.
\end{align*}
\begin{proposition} \label{prop:orthogonalitySchatzWahlbin}
Let $D_1\subset\subset D_2\subset\subset D$ be given. Moreover, let $f \in C(\bar D_2, \mathbb{R}) \cap H^1(D_2, \mathbb{R})$ and $f_h \in \widetilde V_h$ satisfy
\begin{align} \label{eq:prop:orthogonalitySchatzWahlbin}
\left( \nabla (f-f_h), \nabla v_h \right)_{L^2(D_2)} - \left( \frac{\zeta}{c} \right)^2 \left( f-f_h, v_h \right)_{L^2(D_2)} = 0
\end{align}
for all $v_h\in \widetilde V_h \text{ with }\operatorname{supp}v_h \subseteq D_2$. Then there exist constants $C,C^\prime>0$ such that if $\operatorname{dist}(D_1,\partial D_2) \geq \kappa$, $C^\prime h \leq \kappa$ and $\operatorname{dist}(D_2,\partial D) \geq \kappa$, then for $0\leq l \leq 2$ 
\begin{align*}
\lVert f - f_h \rVert_{L^\infty(D_1)} \leq & C \left( |\ln h|h^l \lVert f \rVert_{W^{l,\infty}(D_2)} + \kappa^{-1} \lVert f-f_h \rVert_{L^1(D_2)} \right).
\end{align*}
The constants $C,C^\prime$ do not depend on $h,f,f_h,D_1$ and $D_2$.
\end{proposition}

\begin{proof}
All the assumptions in \cite[Theorem 5.1, Corollary 5.1]{SchatzWahlbinMaxNormEst} are satisfied by the remarks following \cite[A.4]{SchatzWahlbinMaxNormEst}. \cite[Corollary 5.1]{SchatzWahlbinMaxNormEst} requires \eqref{eq:prop:orthogonalitySchatzWahlbin} to hold on $D$. But inspecting the proof of \cite[Corollary 5.1]{SchatzWahlbinMaxNormEst} and the application of \cite[Theorem 5.1]{SchatzWahlbinMaxNormEst} therein shows that \eqref{eq:prop:orthogonalitySchatzWahlbin} is sufficient. Finally, we choose $p=0$, $q = 1$ and have $r=2$ in the statement of \cite[Corollary 5.1]{SchatzWahlbinMaxNormEst}.
\end{proof}

\begin{proof}[proof of Proposition \ref{lem:pointwise_green_disc_error}]
Let $z \in M_\kappa$ and let $x \in D_\kappa$. Theorem \ref{thm:discretized_hh_has_solution} shows that there exists a $h_0$ independent of $x$ such that $G^x_h$ exists for $h\in (0,h_0]$. We want to apply Proposition \ref{prop:orthogonalitySchatzWahlbin} to the real and imaginary part of $G^x-G^x_h \in C(\bar D_{\frac{\kappa}{2}}) \cap H^1(D_{\frac{\kappa}{2}})$ to obtain
\begin{align} \label{eq:continuous_umf}
\begin{split}
|G^x(z)- G^x_h(z)| \leq \lVert G^x - G^x_h \rVert_{L^\infty(B_{\frac{1}{8}\kappa}(z))} 
\leq  C \left( |\ln h| h^2 \lVert G^x \rVert_{ W^{2,\infty}( B_{\frac{1}{4} \kappa}(z) )} +  \lVert G^x-G^x_h \rVert_{L^1(B_{\frac{1}{4}\kappa}(z))} \right).
\end{split}
\end{align}
Here $C$ does not depend on the used balls or on the particular $z$, but only on $\kappa$. To proof this we choose $D_1 = B_{\frac{1}{8} \kappa}(z)$ and $D_2 = B_{\frac{1}{4} \kappa}(z)$
and apply Proposition \ref{prop:Greensfct_exist_regul} to obtain
\begin{equation*}
\operatorname{Re}(G^x), \operatorname{Im}(G^x) \in C(\overline{B_{\frac{1}{4}\kappa}(z)}, \mathbb{R}) \cap H^1(B_{\frac{1}{4}\kappa}(z), \mathbb{R}).
\end{equation*}
Linearity of real part and imaginary part in the weak formulation yields, after short computation, that for any $v_h \in \tilde V_h$ with $\operatorname{supp} v_h \subseteq D_2$ there holds
\begin{align*}
\left( \nabla (\operatorname{Re}(G^x - G^x_h)), \nabla v_h \right)_{L^2(D_2)} - \left( \frac{\zeta}{c} \right)^2 \left(\operatorname{Re}(G^x - G^x_h), v_h \right)_{L^2(D_2)} = v_h(x) - v_h(x) = 0,
\end{align*}
and
\begin{align*}
\left( \nabla (\operatorname{Im}(G^x - G^x_h)), \nabla v_h \right)_{L^2(D_2)} 
- \left( \frac{\zeta}{c} \right)^2 \left(\operatorname{Im}(G^x - G^x_h), v_h \right)_{L^2(D_2)} = 0 - 0 = 0.
\end{align*}
This proves \eqref{eq:continuous_umf} according to Proposition \ref{prop:orthogonalitySchatzWahlbin}. By Proposition \ref{prop:Greensfct_exist_regul} we obtain
\begin{align} \label{eq:proof:pointwise_green_disc_error1}
\lVert G^x \rVert_{W^{2,\infty}(B_{\frac{1}{4}\kappa}(z))} \leq C_\kappa,
\end{align}
with $C_\kappa$ depending on $\kappa$ but not on $x$ or $z$. This implies
\begin{align*}
|G^x(z)- G^x_h(z)| \leq C \left( |\ln h|h^2 + \lVert G^x-G^x_h \rVert_{L^1(B_{\frac{1}{4}\kappa}(z))} \right).
\end{align*}
$C$ does not depend on $x$ or $z$. $\lVert G^x-G^x_h \rVert_{L^1(B_{\frac{1}{4}\kappa}(z))}$ is estimated as in the proof of \cite[Theorem 4.4]{[BeGaRo]}, which is based on the proof of \cite[Theorem 6.1]{SchatzWahlbinMaxNormEst} and \cite[Lemma 3.5]{[BeGaRo]}. Tracking the constants in both proofs shows the independence of $x$ and $z$.

\end{proof}

\subsubsection*{Proof of Theorem \ref{thm:mcmc}} 

For ease of notation we drop the dependence of the measures on the fixed $k$. We prove $\mu^y$-invariance by verifying the detailed balance condition
\begin{equation} \label{eq:detailed_balance}
a(u, u^\prime) \eta(du, du^\prime) = a(u^\prime, u) \eta(du^\prime, du),
\end{equation}
with $\eta(du, du^\prime) := q(u, du^\prime) \mu^y(du)$, see \cite[Section 5.2.]{Dashti2017}. 
%
We first show the reversibility of the proposal with respect to the prior
\begin{equation} \label{eq:reversibility_prior}
q(u, du^\prime) \mu^0(du) = q(u^\prime, du) \mu^0(du^\prime).
\end{equation}
We use the independence of the positions and amplitudes in the proposal and the prior to obtain
\begin{align*}
q(u, du^\prime) \mu^0(du) = q_\alpha(\alpha, d\alpha^\prime) \mu^0_\alpha(d\alpha) \cdot	q_x(x, dx^\prime) \mu^0_x(dx).
\end{align*} 
Here $q_x$ denotes the proposal associated with the positions and $q_\alpha$ the proposal associated with $\alpha$ (see \eqref{eq:proposal}). For the positions we have
\begin{align*}
q_x(x, dx^\prime) \mu^0_x(dx) &= 1_{D_\kappa}(x^\prime) \, N(x, \gamma_x I)(x^\prime) dx^\prime \, \mu^0_x(dx) + \delta_x(dx^\prime) \int_{\mathbb{R}^d} 1_{\mathbb{R}^d \setminus D_\kappa}(z) N(x, \gamma_x I)(z) dz \, \mu^0(dx).
\end{align*} 
The first term describes the probability of the proposal $x + \gamma_x \eta$ lying in $D_\kappa$ and being chosen as $x^\prime$. The second refers to the situation that $x + \gamma_x \eta$ does not lie in $D_\kappa$ and thusly $x^\prime = x$ is chosen. The second summand is clearly symmetric due to the Dirac measure. The first summand is symmetric by using $N(x,\gamma_x I)(x^\prime) = N(x^\prime,\gamma_x, I)(x)$ and $\mu_x^0(dx) = |D_\kappa|^{-1} 1_{D_\kappa}(x) dx$.

For the amplitudes observe that $q_\alpha(\alpha, d\alpha^\prime) \mu_\alpha^0(d\alpha)$ is a complex Gaussian measure on $\mathbb{C}^k \times \mathbb{C}^k = \mathbb{C}^{2k}$ with mean $(m_\alpha, m_\alpha)$, covariance matrix $\Gamma_{\alpha, \alpha}$ matrix and relation $C_{\alpha, \alpha}$ given by
\begin{align*}
\Gamma_{\alpha, \alpha} = \begin{pmatrix}
\Gamma & (1 - \gamma^2_{\alpha})^{1 / 2} \Gamma \\
(1 - \gamma^2_{\alpha})^{1 / 2} \Gamma & \Gamma
\end{pmatrix}, ~
C_{\alpha, \alpha} = \begin{pmatrix}
C & (1 - \gamma^2_{\alpha})^{1 / 2} C \\
(1 - \gamma^2_{\alpha})^{1 / 2} C & C
\end{pmatrix}.
\end{align*}
This measure is symmetric with respect to the first and last $k$ amplitudes. Hence we conclude
\begin{align*}
q_\alpha(\alpha, d\alpha^\prime) \mu^0_\alpha(d\alpha) &= q_\alpha(\alpha^\prime, d\alpha) \mu^0_\alpha(d\alpha^\prime) \\
q_x(x, dx^\prime) \mu^0_x(dx) &= q_x(x^\prime, dx) \mu^0_x(dx^\prime)
\end{align*}
and thus we have shown \eqref{eq:reversibility_prior}. A straight forward computation now shows 
\begin{align*}
\eta(du, du^\prime) &= q(u, du^\prime) \mu^y(du) = \frac{d\mu^y}{d\mu^0}(u) q(u, du^\prime) \mu^0(du) = \frac{d\mu^y}{d\mu^0}(u) q(u^\prime, du) \mu^0(du^\prime) \\
 &= \frac{d\mu^y}{d\mu^0}(u) \frac{d\mu^0}{d\mu^y}(u^\prime) q(u^\prime, du) \mu^y(du^\prime) = \exp(\Psi(u^\prime) - \Psi(u)) \eta(du^\prime, du).
\end{align*}
Thus the detailed balance \eqref{eq:detailed_balance} is satisfied with $a$ defined as in Theorem \ref{thm:mcmc}. 

\subsubsection*{Derivation of \eqref{eq:hellinger_numerics}} \label{sec:appendix_hellinger}
The square of the Hellinger Distance is defined as
\begin{equation*}
d^2_{\text{Hell}}(\mu^y, \mu^y_h) = \int_X \frac{1}{2}\left(\left(\frac{d\mu^y}{d\nu}\right)^{\frac{1}{2}} - \left(\frac{d\mu^y_h}{d\nu}\right)^{\frac{1}{2}}\right)^2 d\nu
\end{equation*}
and both $\mu^y_h, \mu^y$ are absolutely continuous w.r.t. the reference measure $\nu$. The measures $\mu^y_h$ and $\mu^y$ are equivalent since
\begin{equation*}
d \mu^y(u) \propto \exp(-\Phi(u)) d \mu^0 \propto \exp(\Phi_h(u) -\Phi(u)) \exp(- \Phi_h(u)) d \mu^0 \propto \exp(\Phi_h(u) -\Phi(u)) d\mu^y_h.
\end{equation*}
Hence we conclude
\begin{align*}
2 d^2_{\text{Hell}}(\mu^y, \mu^y_h) = \int_X \left(1 - \left(\frac{d\mu^y_h}{d\mu^y}\right)^{\frac{1}{2}}\right)^2 d\mu^y, \qquad
2 d^2_{\text{Hell}}(\mu^y, \mu^y_h) = \int_X \left(\left(\frac{d\mu^y}{d\mu^y_h}\right)^{\frac{1}{2}} - 1 \right)^2 d\mu^y_h.
\end{align*}
We add these equations to obtain \eqref{eq:hellinger_numerics}.
\end{appendix}

\FloatBarrier


\end{document}